\documentclass{article}

\usepackage[usenames,dvipsnames]{pstricks}
\usepackage{epsfig}
\usepackage{pst-grad} 
\usepackage{pst-plot} 
\usepackage[space]{grffile} 
\usepackage{etoolbox} 
\makeatletter 
\patchcmd\Gread@eps{\@inputcheck#1 }{\@inputcheck"#1"\relax}{}{}
\makeatother

\usepackage[overload]{empheq}

 \usepackage{mathtools}

\usepackage{amsmath}
\usepackage{amssymb}
\usepackage{amsthm}
\usepackage[latin1]{inputenc}
\usepackage{graphicx}

\usepackage[T1]{fontenc}

\usepackage{ifthen}
\usepackage{xcolor}

\newcommand{\intav}[1]{\mathchoice {\mathop{\vrule width 6pt height 3 pt depth  -2.5pt
\kern -8pt \intop}\nolimits_{\kern -6pt#1}} {\mathop{\vrule width
5pt height 3  pt depth -2.6pt \kern -6pt \intop}\nolimits_{#1}}
{\mathop{\vrule width 5pt height 3 pt depth -2.6pt \kern -6pt
\intop}\nolimits_{#1}} {\mathop{\vrule width 5pt height 3 pt depth
-2.6pt \kern -6pt \intop}\nolimits_{#1}}}

\def\polhk#1{\setbox0=\hbox{#1}{\ooalign{\hidewidth\lower1.5ex\hbox{`}\hidewidth\crcr\unhbox0}}}

\renewcommand{\div}{\operatorname{div}}

\renewcommand{\div}{\operatorname{div}}

\newtheorem{teo}{Theorem}

\newtheorem{Definition}{Definition}
\newtheorem{Lemma}{Lemma}
\newtheorem{Corollary}{Corollary}
\newtheorem{Proposition}{Proposition}
\newtheorem{Remark}{Remark}
\newtheorem{Assumption}{A}

\usepackage{setspace}
\begin{document}

\title{A fully nonlinear degenerate free transmission problem}
\author{Gerardo Huaroto, Edgard A. Pimentel, Giane C. Rampasso \\and Andrzej \'{S}wi\k{e}ch}

\date{\today} 

\maketitle

\begin{abstract}
\noindent We study a free transmission problem driven by degenerate fully nonlinear operators. Our first result concerns the existence of a viscosity solution to the associated Dirichlet problem. By framing the equation in the context of viscosity inequalities, we prove regularity results for the constructed viscosity solution to the problem. Our findings include regularity in $ C^{1,\alpha}$ spaces, and an explicit characterization of $\alpha$ in terms of the degeneracy rates. We argue by perturbation methods, relating our problem to a homogeneous, fully nonlinear uniformly elliptic equation. 

\medskip


\medskip 

\end{abstract}

\vspace{.1in}

\section{Introduction}\label{sec_introduction}

We examine viscosity solutions $u\in C(\overline\Omega)$ to the free transmission problem
\begin{equation}\label{eq_main}
	\begin{split}
		\left|Du\right|^{\theta_1}F(D^2u)&=f(x)\hspace{.2in}\mbox{in}\hspace{.2in}\Omega^+(u)\cap\Omega\\
		\left|Du\right|^{\theta_2}F(D^2u)&=f(x)\hspace{.2in}\mbox{in}\hspace{.2in}\Omega^-(u)\cap\Omega,
	\end{split}
\end{equation}
where $\Omega$ is a bounded domain in $\mathbb{R}^d$, $\Omega^+(u):=\{u>0\}$, and $\Omega^-(u):=\{u<0\}$. In addition, $F:S(d)\to\mathbb{R}$ is a $(\lambda,\Lambda)$-elliptic operator, $\theta_i>0, i=1,2,$ are fixed constants, the source term $f:\Omega\to\mathbb{R}$ is a bounded function and $S(d)$ stands for the space of $d\times d$ symmetric matrices. We prove the existence of viscosity solutions to the Dirichlet problem associated with \eqref{eq_main} and establish optimal regularity in $ C^{1,\alpha}$-spaces, with appropriate estimates. 

The model in \eqref{eq_main} accounts for a diffusion process degenerating as a power of the gradient. The degeneracy law depends on the sign of the solution, which introduces discontinuities along $\partial\{u>0\}$ and $\partial\{u<0\}$. We emphasize the subregions where distinct degeneracy regimes take place are \emph{unknown a priori}, and depend on the solution. Therefore, the transmission interface is understood as a free boundary. We notice there is no a priori reason for $\{u=0\}=\emptyset$, or even for the zero level set of $u$ to have null measure in $\Omega$; as a consequence, \eqref{eq_main} prescribes a PDE only in a subregion of 
$\Omega$. 

Transmission problems account for diffusion processes in heterogeneous media, with applications to thermal and electromagnetic conductivity and composite materials, such as fiber-reinforced structures. A typical formulation can be described as follows. Given a domain $\Omega\subset\mathbb{R}^d$, we choose $k-1$ open, mutually disjoint subregions $\Omega_i\Subset\Omega$, for $i=1,\,\ldots,\,k-1$, and set
\[
	\Omega_k\,:=\,\Omega\setminus\bigcup_{i=1}^{k-1}\overline{\Omega_i}.
\]
Inside each $\Omega_i$ a different equation is prescribed. For example, let $A:\Omega\to\mathbb{R}^{d^2}$ be a matrix-valued mapping and consider
\[
	\div\left(A(x)Du\right)\,=\,0\hspace{.4in}\mbox{in}\hspace{.1in}\Omega,
\]
where 
\[
	A(x)\,=\,A_i,
\]
for $x\in\Omega_i$, where $A_i$ are constant matrices and $i=1,\,\ldots,\,k$. Within each subregion, $u$ solves a divergence-form equation governed by constant coefficients. However, across the transmission interface $\partial\Omega_i$, the diffusion process may be discontinuous. 

Those discontinuities introduce difficulties in the study of the problem, affecting the understanding of properties such as existence and uniqueness of solutions and their regularity. Here, an important aspect of the analysis is the geometry of the transmission surfaces.

The first treatment of this class of problems appeared in \cite{Picone1954}, and was followed by a number of developments, pursued by several authors; we mention \cite{Borsuk1968,Campanato1957,Campanato1959,Campanato1959a,Iliin-Shismarev1961,Lions1955,Oleinik1961,Schechter1960,Stampacchia1956,Sheftel1963}. The findings reported in these papers concern the well-posedness of transmission problems in distinct settings. For a comprehensive account of these results, we refer the reader to \cite{Borsuk2010}.

The regularity of solutions to transmission problems has also been investigated in the literature. In \cite{Li-Vogelius2000} the authors consider a bounded domain $\Omega$ in the presence of a finite number of subregions $(\Omega_i)_{i=1}^k$ which are known a priori. In the interior of each sub-region an equation in the divergence-form holds. Under regularity assumptions on the diffusion coefficients and the geometry of the transmission interface, the authors prove that solutions are $ C^{1,\alpha}$-regular, locally; their estimates do not depend on the proximity of the sub-regions (compare with \cite{Bonnetier2000}). They also connect this model with the analysis of composite materials with closely spaced inclusions; the typical object here is a fiber-reinforced structure. In this context, the gradient relates to the stress of the material. As a consequence, their findings provide information on the qualitative behavior of this quantity as well. Of particular interest is the independence of the estimates of the location of the sub-regions: here, it translates into a result \emph{independent of the location of the fibers}. A vectorial counterpart of those results is reported in \cite{Li-Nirenberg2003}. In that paper, the authors also derive bounds on higher derivatives of the solutions, under additional conditions on the data of the problem.

In \cite{Bao-Li-Yin2009,Bao-Li-Yin2010} the authors examine a transmission problem related to the theory of conductivity. The model in \cite{Bao-Li-Yin2009} consists of a bounded domain with two compactly contained sub-regions. These are $\varepsilon$-apart, where $\varepsilon>0$ is a parameter. Inside each sub-region, the divergence-form equation governing the problem has constant coefficient $k>0$. In the complementary region the diffusion coefficient is taken to be equal to $1$. The authors consider the case of perfect conductivity $(k=+\infty)$ and examine the behavior of gradient bounds for the solutions as $\varepsilon\to 0$. Though it is known that such bounds deteriorate as both subregions approach each other, the findings of \cite{Bao-Li-Yin2009} produce a blow-up rate for those estimates. In \cite{Bao-Li-Yin2010} the authors examine
the case of multiple subregions -- standing for multiple inclusions -- and consider also the case of insulation $(k=0)$. We also mention \cite{Briane-Capdeboscq-Nguyen2013}.

A typical bottleneck in the regularity theory for solutions to transmission problems is the geometry of the transmission interface. In \cite{Caffarelli-Carro-Stinga2020}, the authors consider a domain $\Omega\subset \mathbb{R}^d$ and a subregion $\Omega_1\Subset\Omega$. By prescribing $\Omega_1$, they also define $\Omega_2:=\Omega\setminus\Omega_1$. The main contribution of that paper is in the fact that $\partial\Omega_1$ is supposed to be merely of class $ C^{1,\alpha}$. Under this assumption, and a balance condition relating the normal derivatives of the solutions across the transmission interface, the authors prove $ C^{1,\alpha}$-regularity of solutions to a problem driven by the Laplace operator. Their arguments rely on a new stability result. The latter connects the transmission problem under analysis to an auxiliary model, with flat interfaces.

The developments mentioned so far concern transmission problems in which the sub-regions of interest are known a priori. However, a natural generalization regards the case where those subregions depend on the solution and, as a consequence, are endogenously determined. In this case, transmission problems can be framed in the context of free boundary analysis. This is precisely the context in the present work. Owing to the fact that the transmission interface behaves as a free boundary, this class of models is referred to as \emph{free transmission problems}.

In \cite{Amaral-Teixeira2015}, the authors consider a free transmission problem governed by the minimization of the functional
\begin{equation}\label{eq_func}
	I(v)\,:=\,\int_{\Omega}\left(\frac{1}{2}\left\langle A(x,v)Dv,\,Dv\right\rangle\,+\,\Lambda(v)\,+\,fv\right)\,dx,
\end{equation}
where
\[
	\begin{split}
		&A(x,u)\,:=\,A_+(x)\chi_{\left\lbrace u>0\right\rbrace}\,+\,A_-(x)\chi_{\left\lbrace u\leq0\right\rbrace},\\
		&\Lambda(u)\,:=\,\lambda_+(x)\chi_{\left\lbrace u>0\right\rbrace}\,+\,\lambda_-(x)\chi_{\left\lbrace u\leq0\right\rbrace},\\
		&f\,:=\,f_+(x)\chi_{\left\lbrace u>0\right\rbrace}\,+\,f_-(x)\chi_{\left\lbrace u\leq0\right\rbrace},
	\end{split}
\]
and $A_\pm:\Omega\to\mathbb{R}^{d^2}$ are matrix-valued mappings satisfying suitable ellipticity conditions, $f_\pm:\Omega\to\mathbb{R}$ are source terms in appropriate Lebesgue spaces and $\lambda_\pm:\Omega\to\mathbb{R}$ encode balance conditions of the model. The critical points of \eqref{eq_func} satisfy the divergence form equations
\begin{equation}\label{eq_amtei}
	\begin{cases}
		\div(A_+(x)Du(x))\,=\,f_+&\hspace{.4in}\mbox{in}\hspace{.1in}\{u>0\}\cap\Omega\\
		\div(A_-(x)Du(x))\,=\,f_-&\hspace{.4in}\mbox{in}\hspace{.1in}\{u\leq0\}^\circ\cap\Omega,
	\end{cases}
\end{equation}
equipped with a flux condition across the free transmission interface $\partial\{u>0\}\cap\Omega$; the latter is derived through a Hadamard-type argument. The authors establish the existence of minimizers for \eqref{eq_func}, with uniform estimates in $L^\infty(\Omega)$. Notice that the functional under analysis lacks convexity, which entails further difficulties in the analysis. Moreover, they prove that local minima have a modulus of continuity. Finally, the authors resort to a perturbation argument and suppose $A_+$ and $A_-$ to be close, in a suitable topology. Under those conditions, solutions to \eqref{eq_amtei} are proved to be asymptotically Lipschitz. 

The model under analysis in the present paper is a free transmission problem. Indeed, the regions where distinct degeneracy laws hold depend on the solution itself. In addition, the problem has a fully nonlinear, \emph{non-variational}, structure. It means there is no underlying minimization problem providing information about \eqref{eq_main}. Finally, the equation is allowed to degenerate, and does so as the gradient of the solutions vanishes. 

Fully nonlinear equations degenerating as a power of the gradient have been examined in various contexts; the work-horse of the theory takes the form
\begin{equation}\label{eq_power}
	|Du|^\theta F(D^2u)\,=\,f(x)\hspace{.4in}\mbox{in}\hspace{.1in}\Omega,
\end{equation}
where $F$ is a fully nonlinear uniformly elliptic operator, $\theta>-1$ is a constant, and $f\in L^\infty(\Omega)\cap C(\Omega)$. This is modeled as a non-variational, fully nonlinear, variant of the $p$-Poisson equation. Among the results available for the solutions to \eqref{eq_power}, we mention comparison and maximum principles, well-posedness for the Dirichlet problem, and an Aleksandroff-Bakelman-Pucci estimate; we refer the reader to \cite{Birindelli-Demengel2004}, \cite{Birindelli-Demengel2006}, \cite{Birindelli-Demengel2007A}, \cite{Birindelli-Demengel2007B}, \cite{Birindelli-Demengel2008}, \cite{Gonzalo-Felmer-Quaas2009}, and the references therein.

The regularity of solutions to \eqref{eq_power} is the subject of \cite{Imbert-Silvestre2013}, \cite{Birindelli-Demengel2014} and \cite{Araujo-Ricarte-Teixeira2015}. If $u\in C(\Omega)$ is a viscosity solution to \eqref{eq_power}, then $u\in C^{1,\alpha}_{\rm loc}(\Omega)$, with the appropriate estimates. In particular, the H\"older-exponent $\alpha$ satisfies
\[
	\alpha\,\in\left(0,\alpha_0\right),\hspace{.2in}\alpha\leq\frac{1}{1\,+\,\theta},
\]
where $\alpha_0$ stands for the exponent in the regularity theory associated with the homogeneous equation $F=0$; see Proposition \ref{th:alpha0} and, for instance, \cite[Section 5.3]{Caffarelli-Cabre1995}. If $F$ is convex, then $\alpha_0=1$ and solutions are precisely $ C^{1,\frac{1}{1+\theta}}$-regular, locally. 

The case of variable exponents $\theta=\theta(x)$ is the topic of \cite{Bronzi-Pimentel-Rampasso-Teixeira2018}. In that paper the authors examine the regularity of viscosity solutions under the assumption that $\theta:\Omega=B_1\to\mathbb{R}$ is a \emph{continuous} function. 

Thus the analysis of \eqref{eq_main} does not fall into the scope of that paper. Moreover, to the best of our knowledge, an $L^p$-viscosity theory under structural conditions accommodating \eqref{eq_power} is not available. See \cite{CCKS1996}; see also \cite{Koike2006}.

In \cite{imbert_silvestre_jems} the authors study elliptic equations holding only in regions where the \emph{gradient of the solutions is large}. The problem is formulated in terms of structural conditions involving the extremal Pucci operators and continuous, bounded, ingredients. In that setting, they prove that viscosity solutions are H\"older-continuous. Their argument involves a measure estimate associated with touching the solutions with cusps and examining the contact points. In \cite{mooney}, the author extends the results of \cite{imbert_silvestre_jems} to the context of measurable ingredients. By touching the graph of the solutions with paraboloids instead of cusps, and exploring a dyadic decomposition algorithm, the author proves a Harnack inequality and H\"older-continuity for the strong solutions to an equation holding only at the points where the gradient is large. The findings in \cite{mooney} are reported in the context of a linear elliptic operator for simplicity, as the results hold in the fully nonlinear setting as well.

In this article we obtain existence of viscosity solutions to a Dirichlet boundary value problem for \eqref{eq_main} (that is to equation \eqref{eq_deg0}) and study its regularity. To prove existence of viscosity solutions to \eqref{eq_deg0}, we consider approximating problems for which solutions are obtained by a fixed point argument. After this is done, a limiting procedure produces a viscosity solution to \eqref{eq_deg0}, in a suitable sense. The procedure to obtain a viscosity solution to \eqref{eq_deg0} guarantees that the solution is a viscosity subsolution and a viscosity supersolution to certain associated extremal differential equations on the whole $\Omega$. This fact is the key in obtaining the regularity of our viscosity solution. A uniform exterior sphere condition for $\Omega$ can be replaced by a different condition guaranteeing existence of suitable barriers, however we assume it to avoid unnecessary technicalities since the barriers are more explicit. Our existence result is the following.

\begin{teo}[Existence of solutions] \label{thm_existence}
Let $\Omega\subset \mathbb{R}^d$ be a bounded domain which satisfies a uniform exterior sphere condition. Suppose Assumptions A\ref{assump_Felliptic} and A\ref{assump_theta} (see below) hold, $f\in C(\overline\Omega)$ and $g\in C(\partial\Omega)$. Then there exists a viscosity solution $u\in C(\overline\Omega)$ to 
\begin{equation}\label{eq_deg0}
	\begin{cases}
			\left|Du\right|^{\theta_1}F(D^2u)=f(x)&\hspace{.2in}{\rm{in}}\hspace{.2in}\Omega^+(u)\cap\Omega\\
			\left|Du\right|^{\theta_2}F(D^2u)=f(x)&\hspace{.2in}{\rm{in}}\hspace{.2in}\Omega^-(u)\cap\Omega\\
			u=g&\hspace{.2in}{\rm{on}}\hspace{.2in}\partial\Omega.
	\end{cases}
\end{equation}
The solution $u$ obtained is a viscosity subsolution to
\begin{equation}\label{eq_cviscsub}
	\min\left\lbrace |Du|^{\theta_2}F(D^2u),\,F(D^2u)\right\rbrace\,=\,C_0 \quad\mbox{in}\hspace{.1in}\Omega
\end{equation}
and a viscosity supersolution to
\begin{equation}\label{eq_cviscsuper}
	\max\left\lbrace |Du|^{\theta_2}F(D^2u),\,F(D^2u)\right\rbrace\,=\,-C_0\quad\mbox{in}\hspace{.1in}
	\Omega,
\end{equation}
where $C_0=\|f\|_{L^\infty(\Omega)}$.
\end{teo}

Once the existence of solutions is addressed, we examine their regularity. From this point on we assume for simplicity that $\Omega=B_1$.
We observe that the $ C^{1,\alpha}$-regularity only relies on the two differential
inequalities \eqref{eq_cviscsub}-\eqref{eq_cviscsuper}. The $ C^{1,\alpha}$ result is based
on a perturbation argument which relates \eqref{eq_cviscsub}-\eqref{eq_cviscsuper} with equation $\overline F=0$ for some $\overline F$ having the same structure as $F$. Our main theorem is the following.

\begin{teo}[H\"older-regularity of the gradient]\label{teo_main2}
Assume $\Omega=B_1$. Let $u\in  C(B_1)$ be a viscosity subsolution to \eqref{eq_cviscsub} and a a viscosity supersolution to \eqref{eq_cviscsuper} for some $C_0\geq 0$. Suppose Assumptions A\ref{assump_Felliptic} and A\ref{assump_theta} (see below) hold. Let
\begin{equation}\label{eq:alpharange}
	\alpha\,\in\,(0, \alpha_0),\,\,\alpha\leq \frac{1}{1\,+\,\theta_2}.
\end{equation}
Then $u\in  C^{1,\alpha}_{\rm loc}(B_1)$ and, for every $0<\tau<1$, there exists $C>0$ such that
\begin{equation}\label{eq:calphaC0}
	\left\|u\right\|_{ C^{1,\alpha}(B_\tau)}\leq C\left(\left\|u\right\|_{L^\infty(B_1)}+\max\left\lbrace C_0,C_0^\frac{1}{1+\theta_2}\right\rbrace\right),
\end{equation}
where $C=C(d,\lambda,\Lambda,\theta_2,\alpha,\tau)$.
\end{teo}

If $F$ is a convex operator, the Evans-Krylov theory for equation $F=0$ becomes available giving $\alpha_0=1$. Then, Theorem \ref{teo_main2} produces an \emph{optimal regularity result} and viscosity solutions to a pair of differential inequalities \eqref{eq_cviscsub}-\eqref{eq_cviscsuper} are of class $ C^{1,\alpha^*}_{\rm loc}(B_1)$, where
\[
	\alpha^*\,:=\,\frac{1}{1\,+\,\theta_2},
\]
with appropriate estimates.

\begin{Remark}\normalfont
Consider the problem
\begin{equation}\label{eq_everywhere}
	|Du|^{\theta_1\chi_{\{u>0\}}+\theta_2\chi_{\{u<0\}}}F(D^2 u)=f\hspace{.2in}\mbox{in}\hspace{.2in}\Omega.
\end{equation}
It differs from \eqref{eq_main} in the sense that it prescribes the equation in the entire domain. The proof of Theorem \ref{thm_existence} also yields the existence of a viscosity solution $u$ to \eqref{eq_everywhere} in the standard sense on $\Omega^+(u)\cup\Omega^-(u)$, however on the set $\{u=0\}$, $u$ is only a viscosity subsolution to
\[
	\min\left\lbrace|Du|^{\theta_2}F(D^2u),F(D^2u)\right\rbrace=f(x)
\]
and a viscosity supersolution to
\[
	\max\left\lbrace|Du|^{\theta_2}F(D^2u),F(D^2u)\right\rbrace=f(x).
\]
Hence, the conclusion of Theorem \ref{teo_main2} is also available in this setting. Moreover it is also straightforward that a standard viscosity solution $u\in C(\Omega)$ to an equation $|Du|^{\theta(x)}F(D^2u)=f(x)$, where $\theta$ is defined everywhere on $\Omega$ but may be discontinuous, is a viscosity subsolution to \eqref{eq_cviscsub} and a viscosity supersolution to \eqref{eq_cviscsuper}. Therefore, the regularity result of Theorem \ref{teo_main2} also applies to such equations.
\end{Remark}

The remainder of this paper is organized as follows. In Section \ref{subsec_ma} we describe the main assumptions of the paper. In Section \ref{subsec_prelim} we recall some elementary notions and show preliminary results, whereas in Section \ref{subsec_smallness} we discuss scaling properties of the model. The proof of Theorem \ref{thm_existence} is in Section  \ref{sec_existence}. Section \ref{sec_improvedreg}
is devoted to the proof of Theorem \ref{teo_main2}.

\section{Preliminaries}\label{sec_pmma}

In this section we collect a few notions, known and preliminary results, and the assumptions under which we will work in this article. 

\subsection{Main assumptions}\label{subsec_ma}

Our main assumptions concern the uniform ellipticity of the fully nonlinear operator governing \eqref{eq_main} and
the degeneracy degree constants $\theta_1$ and $\theta_2$. 

\begin{Assumption}[Uniform ellipticity]\label{assump_Felliptic}
The function $F:S(d)\to\mathbb{R}$ is such that $F(0)=0$ and is $(\lambda,\Lambda)$-uniformly elliptic; that is, for some fixed constants
$0<\lambda\leq\Lambda$,
\[
\lambda\left\|N\right\|\,\leq\,F(M)\,-\,F(M\,+\,N)\,\leq\,\Lambda\left\|N\right\|
\]
for every $M,\,N\in S(d)$, with $N\geq0$. 
\end{Assumption}

\begin{Assumption}[Degeneracy rates]\label{assump_theta}
The constants $\theta_1,\,\theta_2\in\mathbb{R}$ satisfy
\[
	0\,<\,\theta_1\,<\,\theta_2.
\]	
\end{Assumption}

We notice that Assumption \ref{assump_Felliptic} in particular implies that $F$ is degenerate elliptic.

\subsection{Preliminary notions and results}\label{subsec_prelim}

We begin with the definition of the Pucci extremal operators. 

\begin{Definition}[Extremal operators]\label{def_pucci}
Let $0<\lambda\leq\Lambda$ be as in Assumption \ref{assump_Felliptic}. We define the extremal Pucci operators $\mathcal{P}_{\lambda,\Lambda}^\pm:S(d)\to\mathbb{R}$ as follows:
\[
	\mathcal{P}_{\lambda,\Lambda}^+(M)\,:=\,-\Lambda\sum_{e_i<0}e_i\,-\,\lambda\sum_{e_i>0}e_i
\]
and
\[
	\mathcal{P}_{\lambda,\Lambda}^-(M)\,:=\,-\Lambda\sum_{e_i>0}e_i\,-\,\lambda\sum_{e_i<0}e_i,
\]
where $\{e_1,\,e_2,\,\ldots,\,e_d\}$ are the eigenvalues of $M$.
\end{Definition}

We write $\mathcal{P}^\pm_{\lambda,\Lambda}=\mathcal{P}^\pm$, when ellipticity constants have been set. For properties of the extremal operators, we refer the reader to \cite[Section 2.2]{Caffarelli-Cabre1995} or \cite{CCKS1996}. For the sake of completeness, we recall the notion of viscosity solution, see \cite{Crandall-Ishii-Lions1992}. It is the so called ${ C}$-viscosity solution in the terminology of \cite{CCKS1996}.

\begin{Definition}[Viscosity solution]\label{def_cviscosity}
Let $G:\Omega\times\mathbb{R}\times\mathbb{R}^d\times S(d)\to\mathbb{R}$ be a degenerate elliptic operator. We say that an upper semicontinuous function $u:\Omega\to\mathbb{R}$ is a viscosity subsolution to 
\begin{equation}\label{eq_cviscdef}
	G(x,u,Du,D^2u)\,=\,0
\end{equation}
in $U$ if, whenever $\varphi\in C^2(\Omega)$ and $u\,-\,\varphi$ attains a local maximum at $x_0\in \Omega$, we have
\[
	G(x_0,u(x_0),D\varphi(x_0),D^2\varphi(x_0))\,\leq\,0.
\]
Similarly, we say that a lower semicontinuous function $u:\Omega\to\mathbb{R}$ is a viscosity supersolution to \eqref{eq_cviscdef} if, whenever 
$\varphi\in C^2(\Omega)$ and $u\,-\,\varphi$ attains a local minimum at $x_0\in \Omega$, we have
\[
	G(x_0,u(x_0),D\varphi(x_0),D^2\varphi(x_0))\,\geq\,0.
\]
If $u$ is both a viscosity subsolution and supersolution to \eqref{eq_cviscdef}, we say $u$ is a viscosity solution to \eqref{eq_cviscdef}.
\end{Definition}

We recall Perron's method, see e.g. Theorem 4.1 of \cite{Crandall-Ishii-Lions1992}.

\begin{Lemma}[Perron's method]
Let $\Omega$ be a bounded domain and $G\in  C(\Omega\times\mathbb{R}^d\times S(d))$ be degenerate elliptic. Suppose the comparison principle holds for \eqref{eq_cviscdef} with this $G$. Suppose further that there exist a viscosity subsolution $\underline{w}\in C(\overline\Omega)$ and a viscosity supersolution 
$\overline{w}\in C(\overline\Omega)$ of \eqref{eq_cviscdef} such that $\underline{w}\leq\overline{w}$ in $\Omega$ and 
$\underline{w}=\overline{w}$ on $\partial\Omega$ . Then
\[
	u(x):=\left\lbrace u(x)\,\mid\,\underline{w}\leq v\leq\overline{w},\hspace{.02in} \mbox{$v$ is a viscosity subsolution to \eqref{eq_cviscdef}}\right\rbrace
\]
is a viscosity solution to \eqref{eq_deg0}.
\end{Lemma}

We continue by stating the maximum principle for viscosity solutions, Theorem 3.2 of \cite{Crandall-Ishii-Lions1992}.

\begin{Proposition}\label{prop_cil}
Let $\Omega$ be a bounded domain and $H,\,G\in  C(B_1\times \mathbb{R}^d\times S(d))$ be degenerate elliptic. Let $u$ be a viscosity subsolution to $G(x,Du,D^2u)=0$ and $w$ be a viscosity supersolution to $H(x,Dw,D^2w)=0$
in $\Omega$. Let $\psi\in C^2(\Omega\times \Omega)$. Define $v:\Omega \times \Omega\to\mathbb{R}$ by
\[
	v(x,y)\,:=\,u(x)\,-\,w(y).
\]
Suppose further that $(\overline{x},\overline{y})\in \Omega\times \Omega$ is a local maximum of  $v\,-\psi$ in $\Omega\times \Omega$. Then, for each $\varepsilon>0$, there exist matrices $X$ and $Y$ in $\mathcal S(d)$ such that
\[
	G\left(\overline{x},D_x\psi(\overline{x},\overline{y}),X\right)\,\leq\,0\,\leq\,H\left(\overline{y},-D_y\psi(\overline{x},\overline{y}),Y\right),
\]
and the matrix inequality
\begin{equation*}
-\left(\frac{1}{\varepsilon}+\|A\|\right)I\,\leq\,
\left(
\begin{array}{ccc}
X   & 0 \\
0  &-Y 
\end{array}
\right)
\leq 
A+\varepsilon A^2
\end{equation*}
holds true, where $A=D^2\psi\left(\overline{x},\overline{y}\right)$.
\end{Proposition}

When developing perturbation methods we will need compactness properties of the solutions. We will use 
\cite[Theorem 1.1]{imbert_silvestre_jems} which is stated below in a simplified form.

\begin{Proposition}[H\"older-continuity]\label{prop_mooney}
Let $u\in  C(B_1)$ be a bounded viscosity supersolution to 
\[
	\mathcal{P}^+_{\lambda,\Lambda}(D^2u)=-C_0\hspace{.2in}\mbox{in}\hspace{.2in}\left\lbrace|Du|>\gamma\right\rbrace
\]
and a bounded viscosity subsolution to
\[
	\mathcal{P}^-_{\lambda,\Lambda}(D^2u)=C_0\hspace{.2in}\mbox{in}\hspace{.2in}\left\lbrace|Du|>\gamma\right\rbrace,
\]
for some fixed $\gamma>0$, and $C_0\geq 0$. Then $u\in C^\beta_{\rm loc}(B_1)$ and for every $0<\tau<1$ there there exists $C>0$ such that
\[
	\left\|u\right\|_{ C^\beta(B_{\tau})}\,\leq\,C.
\]
The constant $\beta$ depends only on $d$, $\lambda$ and $\Lambda$, and $C$ depends only on $d$, $\lambda$, $\Lambda$, $\gamma$, $\left\|u\right\|_{L^\infty(B_1)}$, $C_0$ and $\tau$.
\end{Proposition}
The above proposition implies H\"older regularity of viscosity solutions to variants of \eqref{eq_cviscsub}-\eqref{eq_cviscsuper}. Indeed, consider $u\in { C}(B_1)$ which is a viscosity subsolution to
\begin{equation}\label{eq_cil1}
	\min\left\lbrace \left|q+Du\right|^{\theta_2}F(D^2u),\,F(D^2u) \right\rbrace = 1
\end{equation}
and a viscosity supersolution to
\begin{equation}\label{eq_cil2}
	\max\left\lbrace \left|q+Du\right|^{\theta_2}F(D^2u),\,F(D^2u) \right\rbrace\ =-1
\end{equation}
in the unit ball $B_1$, where $q\in\mathbb{R}^d$ is an arbitrary vector.
Let $A_0>1$ be such that $|q|<A_0$. Then since $|q+p|>A_0>1$ if $|p|>2A_0$, it is easy to see that $u$ is a viscosity subsolution to
\[
	F(D^2u)=0\hspace{.4in}\mbox{in}\hspace{.1in}\{|Du|>2A_0\}.
\]

As a consequence, in the set $\{|Du|>2A_0\}$, $u$ is a viscosity subsolution to
\[
	\mathcal{P}^-_{\lambda,\Lambda}(D^2u)=0.	
\]
Similarly we obtain that in the set $\{|Du|>2A_0\}$, $u$ is a viscosity supersolution to
\[
	\mathcal{P}^+_{\lambda,\Lambda}(D^2u)=0.
\]
A straightforward application of Proposition \ref{prop_mooney} thus leads to the following corollary.

\begin{Corollary}\label{prop_hrm}
Let $u\in { C}(B_1)$ be a viscosity subsolution to
\eqref{eq_cil1}
and a viscosity supersolution to
\eqref{eq_cil2}. Let Assumptions A\ref{assump_Felliptic}, A\ref{assump_theta} hold and let $\|u\|_{L^\infty(B_1)}\leq 1$. Suppose further that $|q|<A_0$, for some fixed constant $A_0>1$. Then $u\in C^\beta_{\rm loc}(B_1)$ for some $\beta\in(0,1)$ depending only on $d$, $\lambda$, 
$\Lambda$. In addition, for every $0<\tau<1$, there exists $C>0$ such that
\begin{equation}\label{eq:cbeta}
	\left\|u\right\|_{ C^\beta(B_{\tau})}\,\leq\,C
\end{equation}
with $C=C\left(d,\lambda,\Lambda,A_0,\tau\right)$.
\end{Corollary}

We recall the standard $ C^{1,\alpha_0}_{\rm loc}$-regularity result for solutions to $F=0$, see e.g. \cite[Corollary 5.7]{Caffarelli-Cabre1995}. 

\begin{Proposition}\label{th:alpha0}
Let $F$ satisfy Assumption A\ref{assump_Felliptic} and let $h\in C(B_1)$ be a viscosity solution to
\[
	F(D^2h)\,=\,0\hspace{.4in}\mbox{in}\hspace{.1in}B_1.
\]
Then $h\in C^{1,\alpha_0}_{\rm loc}(B_1)$, for some universal constant $\alpha_0\in(0,1)$. Furthermore, there exists $C>0$ depending only on $d$, $\lambda$ and $\Lambda$, such that
\[
	\left\|h\right\|_{ C^{1,\alpha_0}(B_{1/2})}\,\leq\,C\left\|h\right\|_{L^\infty(B_{3/4})}.
\]
\end{Proposition}

%

\subsection{Scaling properties}\label{subsec_smallness}

In this section we examine scaling properties of equations \eqref{eq_cviscsub}
and \eqref{eq_cviscsuper}. Similar properties apply to \eqref{eq_main}. We only discuss scaling about the origin but the procedure can be obviously done about every point with obvious adjustments.

Suppose $u$ is a viscosity subsolution to
\eqref{eq_cviscsub}
and a viscosity supersolution to
\eqref{eq_cviscsuper}
in $B_1$ but in fact we only require that \eqref{eq_cviscsub}-\eqref{eq_cviscsuper} be satisfied in $B_r$ for some $r>0$.
We define for $K>0$, 
\[
	v(x)\,:=\,\frac{u(rx)}{K}.
\]
A straightforward computation implies that in particular $v$ is a viscosity subsolution to 
\[
\min\left\lbrace |Dv|^{\theta_2}\overline F(D^2v),\,\overline F(D^2v)\right\rbrace
=
	\overline C_0\quad\mbox{in}\hspace{.1in}B_{1}
\]
and a viscosity supersolution to
\[
\max\left\lbrace |Dv|^{\theta_2}\overline F(D^2v),\,\overline F(D^2v)\right\rbrace
=
	-\overline C_0\quad\mbox{in}\hspace{.1in}B_{1},
\]
where 
\[
	\overline{F}(M)\,:=\,\frac{r^2}{K}F\left(\frac{K}{r^2}M\right)
\]
and 
\[
	\overline C_0=C_0\max\left(\frac{r^{2+\theta_2}}{K^{1+\theta_2}},\frac{r^2}{K}\right).
\]
Choosing 
\[
	K\,:=\,\left[\left\|u\right\|_{L^\infty(B_1)}\,+\,\max\left\lbrace C_0,C_0^\frac{1}{1+\theta_2}\right\rbrace\right]
\]
and setting $r:=\varepsilon<1$ we obtain
$\|v\|_{L^\infty(B_1)}\leq 1$ and $\overline C_0\leq \varepsilon$. Thus by this kind of scaling we can always assume that viscosity subsolutions/supersolutions $u$ of \eqref{eq_cviscsub}/\eqref{eq_cviscsuper} satisfy $\|u\|_{L^\infty(B_1)}\leq 1$ and $C_0\leq 1$ or $C_0$ is arbitrarily
small.

\section{Existence of solutions}\label{sec_existence}

Next we prove the existence of a viscosity solution to \eqref{eq_deg0} with the required properties. We start by considering an approximating problem and establishing a comparison principle. Let $v\in  C(\overline\Omega)$. For $0<\varepsilon<1$, define the function $g_\varepsilon^v$ as 
\[
	g_\varepsilon^v:=\max\left(\min\left(\frac{v+\varepsilon}{2\varepsilon},1\right),0\right)\quad\mbox{on}\,\,\Omega
\]
and $g_\varepsilon^v=0$ on $\mathbb{R}\setminus\Omega$. Let $\eta_\varepsilon(\cdot)$ be the standard mollifier and consider $h_\varepsilon^v(x):=(g_\varepsilon^v\ast\eta_\varepsilon^v)(x)$, for $x\in \Omega$. Finally, we define the exponent function $\theta_\varepsilon^v:\Omega\to\mathbb{R}$ by setting
\[
	\theta_\varepsilon^v(x):=\theta_1h_\varepsilon^v(x)+(1-h_\varepsilon^v(x))\theta_2.
\]

Notice that $\theta_1\leq \theta_\varepsilon^v\leq \theta_2$. We consider the family of equations
\begin{equation}\label{eq_deg1}
	\left(\varepsilon+|Du|\right)^{\theta_v^\varepsilon(x)}\left[\varepsilon u+F(D^2u)\right]=f(x)\hspace{.4in}\mbox{in}\hspace{.1in}\Omega,
\end{equation}
and prove a comparison principle for its sub and supersolutions.

\begin{Proposition}[Comparison principle]\label{prop_comparison}
Let $\Omega$ be a bounded domain, $F$ be degenerate elliptic and $f\in C(\overline{\Omega})$. Let $u\in USC(\overline\Omega)$ be a viscosity subsolution to \eqref{eq_deg1} and $w\in LSC(\overline\Omega)$ be a viscosity supersolution to \eqref{eq_deg1}. Suppose $u\leq w$ on $\partial\Omega$. Then, $u\leq w$ in $\overline{\Omega}$.
\end{Proposition}
\begin{proof}
If the statement is false, we have $\max_{x\in\overline{\Omega}}(u-w)(x)=:\tau>0$.  For $\delta>0$ we define $\Phi_\delta:\overline{\Omega}\times\overline{\Omega}\to\mathbb{R}$ as
\[
	\Phi_\delta(x,y):=u(x)-w(y)-\frac{|x-y|^2}{2\delta}.
\]
Let $(x_\delta,y_\delta)\in\overline{\Omega}\times\overline{\Omega}$ be such that
\[
	\max_{x,y\in\overline{\Omega}}\Phi_\delta(x,y)=\Phi_\delta(x_\delta,y_\delta)\geq \tau.
\]
We know (see Lemma 3.1 of \cite{Crandall-Ishii-Lions1992}) that 
\begin{equation}\label{eq:a1}
\lim_{\delta\to 0}\frac{|x_\delta-y_\delta|^2}{\delta}=0
\end{equation}
and thus, for small $\delta$, we have $x_\delta,y_\delta\in\Omega$.
From Theorem 3.2 of \cite{Crandall-Ishii-Lions1992} (see also Proposition \ref{prop_cil}), there exist $X,Y\in S(d)$ such that
\[
	\left(\frac{x_\delta-y_\delta}{\delta},X\right)\in\overline{J}^{2,+}u(x_\delta)\hspace{.2in}\mbox{and}\hspace{.2in}\left(\frac{x_\delta-y_\delta}{\delta},Y\right)\in\overline{J}^{2,-}w(y_\delta),
\]
with
\begin{equation}\label{eq_deg3}
	-\frac{3}{\delta}
		\begin{pmatrix}
		I& 0 \\
		0& I
		\end{pmatrix}
			\leq 
		\begin{pmatrix}
		X&0\\
		0&-Y
		\end{pmatrix}
			\leq
	\frac{3}{\delta}
		\begin{pmatrix}
		I& -I \\
		-I& I
		\end{pmatrix},	
\end{equation}
where $I$ is the identity matrix. Inequality \eqref{eq_deg3} implies $X\leq Y$ and, as a consequence of the degenerate ellipticity of $F$, we thus have
for sufficiently small $\delta$,
\begin{equation}\label{eq_deg2}
\frac{\varepsilon\tau}{2}\leq\varepsilon\left(u(x_\delta)-w(y_\delta)\right)\leq \frac{f(x_\delta)}{\left(\varepsilon+\frac{|x_\delta-y_\delta|}{\delta}\right)^{\theta_v^\varepsilon(x_\delta)}}-\frac{f(y_\delta)}{\left(\varepsilon+\frac{|x_\delta-y_\delta|}{\delta}\right)^{\theta_v^\varepsilon(y_\delta)}}.
\end{equation}

Let $|f(x)|\leq C_1$ for all $x\in \Omega$ and let $\omega$ be a modulus of continuity of $f$ on $\overline{\Omega}$. We notice that
\[
\min\left(\left(\varepsilon+\frac{|x_\delta-y_\delta|}{\delta}\right)^{\theta_v^\varepsilon(x_\delta)},\left(\varepsilon+\frac{|x_\delta-y_\delta|}{\delta}\right)^{\theta_v^\varepsilon(y_\delta)}\right)\geq \varepsilon^{\theta_2},
\]
\[
\max\left(-\theta_v^\varepsilon(x_\delta)\ln\left(\varepsilon+\frac{|x_\delta-y_\delta|}{\delta}\right), -\theta_v^\varepsilon(y_\delta)\ln\left(\varepsilon+\frac{|x_\delta-y_\delta|}{\delta}\right)\right)\leq -\theta_2\ln\varepsilon.
\]
Let $C_2$ be the Lipschitz constant of the function $\theta_v^\varepsilon(x)$ and recall that $|\ln (\varepsilon+r)|\leq |\ln\varepsilon|+r$ for $r\geq 0$.
Then
\[
\begin{split}
	&\frac{f(x_\delta)}{\left(\varepsilon+\frac{|x_\delta-y_\delta|}{\delta}\right)^{\theta_v^\varepsilon(x_\delta)}}-\frac{f(y_\delta)}{\left(\varepsilon+\frac{|x_\delta-y_\delta|}{\delta}\right)^{\theta_v^\varepsilon(y_\delta)}}
	\leq 
	\frac{f(x_\delta)-f(y_\delta)}{\left(\varepsilon+\frac{|x_\delta-y_\delta|}{\delta}\right)^{\theta_v^\varepsilon(x_\delta)}}
	\\
	&
	+f(y_\delta)\left(\frac{1}{\left(\varepsilon+\frac{|x_\delta-y_\delta|}{\delta}\right)^{\theta_v^\varepsilon(x_\delta)}}
	-\frac{1}{\left(\varepsilon+\frac{|x_\delta-y_\delta|}{\delta}\right)^{\theta_v^\varepsilon(y_\delta)}}\right)
	\\
	&
	\leq
	\omega(|x_\delta-y_\delta|)\varepsilon^{-\theta_2}+C_1\left|e^{-\theta_v^\varepsilon(x_\delta)\ln\left(\varepsilon+\frac{|x_\delta-y_\delta|}{\delta}\right)}
	-e^{-\theta_v^\varepsilon(y_\delta)\ln\left(\varepsilon+\frac{|x_\delta-y_\delta|}{\delta}\right)}\right|
	\\
	&
	\leq
	\omega(|x_\delta-y_\delta|)\varepsilon^{-\theta_2}+C_1e^{-\theta_2\ln\varepsilon}|\theta_v^\varepsilon(x_\delta)-\theta_v^\varepsilon(y_\delta)|\left|\ln\left(\varepsilon+\frac{|x_\delta-y_\delta|}{\delta}\right)\right|
	\\
	&
	\leq 
	\omega(|x_\delta-y_\delta|)\varepsilon^{-\theta_2}+C_1\varepsilon^{-\theta_2}C_2|x_\delta-y_\delta|\left(|\ln\varepsilon|+\frac{|x_\delta-y_\delta|}{\delta}\right).
\end{split}
\]
Therefore, letting $\delta\to 0$ in \eqref{eq_deg2} and using \eqref{eq:a1}, we obtain $\frac{\varepsilon\tau}{2}\leq 0$, which is a contradiction.
\end{proof}

Once the comparison principle is available for \eqref{eq_deg1}, we examine the existence of viscosity solutions for this equation. To use Perron's method, we construct continuous viscosity sub and supersolutions to \eqref{eq_deg1}, agreeing with $g$ on the boundary $\partial\Omega$.

\begin{Lemma}[Existence of global sub and supersolutions]\label{lem_existepsilon}
Let $\Omega$ be a bounded domain which satisfies a uniform exterior sphere condition. Let Assumptions A\ref{assump_Felliptic} and A\ref{assump_theta} hold and let $f\in C(\overline{\Omega}),g\in C(\partial\Omega)$. Then there exist a viscosity subsolution $\underline{w}\in C(\overline{\Omega})$ to \eqref{eq_deg1} and a viscosity supersolution $\overline{w}\in C(\overline{\Omega})$ to \eqref{eq_deg1} for every $0<\varepsilon<1$ and $v\in  C(\overline\Omega)$, such that
$\underline{w}=\overline{w}=g$ on $\partial\Omega$.
\end{Lemma}
\begin{proof}
We construct a continuous viscosity supersolution $\overline w$ of \eqref{eq_deg1} for every $0<\varepsilon<1$ and $v\in  C(\overline\Omega)$ such that $\overline w=g$ on $\partial\Omega$. We first construct a global supersolution to \eqref{eq_deg1}. Let $\|f\|_{L^\infty(\Omega)}=:K$. We choose a point $x_0$ such that ${\rm dist}(x_0,\Omega)\geq 1$. Denote $K_1:=\max(K,\lambda d)$ and let
\[
w_1(x):=K_2-\frac{K_1}{2\lambda d}|x-x_0|^2,
\]
where $K_2$ is such that $w_1>\|g\|_{L^\infty(\partial\Omega)}$ on $\partial\Omega$. Then for $x\in\Omega$,
\[
	\left(\varepsilon+|Dw_1(x)|\right)^{\theta^\varepsilon_v(x)}\left[\varepsilon w_1(x)+F(D^2w_1(x))\right]\geq K_1\geq f(x).
\]
Let $R>0$; for every $y\in
\partial\Omega$, let $x_y$ be such that $|y-x_y|=R$ and $\overline{B_R(x_y)}\cap\overline \Omega=\{y\}$. Denote $R_1:=R+{\rm diam}(\Omega)$.
Define for $\alpha>2, M>0$,
$w_y(x):=M(R^{-\alpha}-|x-x_y|^{-\alpha})$. Then $w_y(y)=0, w_y(x)>0$ in $\Omega$ and 
\[
Dw_y(x)=M\alpha \frac{x-x_y}{|x-x_y|^{\alpha+2}}
\]
so
\[
|Dw_y|\geq M\alpha\frac{1}{R_1^{1+\alpha}}\quad\mbox{in}\,\,\Omega.
\]
Also
\[
D^2w_y(x)=M\alpha \frac{I}{|x-x_y|^{\alpha+2}}-M\alpha (\alpha+2)\frac{(x-x_y)\otimes(x-x_y)}{|x-x_y|^{\alpha+4}}.
\]
We notice that if $\lambda(\alpha+2)-d\Lambda\geq 1$, then
\[
F(D^2w_y(x))\geq M\alpha \frac{1}{|x-x_y|^{\alpha+2}}(\lambda(\alpha+2)-d\Lambda)
\geq M\alpha \frac{1}{|x-x_y|^{\alpha+2}}.
\]
We now fix $M$ such that
\[
M\alpha\frac{1}{R_1^{1+\alpha}}\geq 1\quad\mbox{and}\quad M\alpha\frac{1}{R_1^{2+\alpha}}\geq K+\|g\|_{L^\infty(\partial\Omega)}.
\]
For $0<\eta<1$ we now define the functions
\[
w_{y,\eta}(x):=g(y)+\eta+C_\eta w_y(x),
\]
where the constants $C_\eta$ are such that $C_\eta\geq 1$ and $w_{y,\eta}\geq g$ on $\partial\Omega$. We notice that the $C_\eta$ only depend 
on the modulus of continuity of $g$ and are independent of $y$. Then, for every $0<\varepsilon,\eta<1$ and $x\in\Omega$
\[
\begin{split}
\left(\varepsilon+|Dw_{y,\eta}(x)|\right)^{\theta^\varepsilon_v(x)}&\left[\varepsilon w_{y,\eta}(x)+F(D^2w_{y,\eta}(x))\right]
\\
&\geq -\|g\|_{L^\infty(\partial\Omega)}+C_\eta M\alpha\frac{1}{R_1^{2+\alpha}}\geq K\geq f(x).
\end{split}
\]
Therefore, the functions $w_{y,\eta}$ are supersolutions of \eqref{eq_deg1} for every $0<\varepsilon,\eta<1$ and $y\in\partial\Omega$. Thus the functions
\[
\tilde w_{y,\eta}(x):=\min(w_{y,\eta}(x),w_1(x))
\]
are viscosity supersolutions of \eqref{eq_deg1}. Finally the function
\[
\overline w(x):=\inf\{\tilde w_{y,\eta}(x):y\in\partial\Omega, 0<\eta<1\}
\]
is the required viscosity supersolution of \eqref{eq_deg1} and $\overline w=g$ on $\partial\Omega$. A viscosity subsolution $\underline w$ of \eqref{eq_deg1} such that $\underline w=g$ on $\partial\Omega$ is constructed similarly.
\end{proof}

The existence of a unique viscosity solution to the approximating equations \eqref{eq_deg1} follows from Lemma \ref{lem_existepsilon}, together with the comparison principle in Proposition \ref{prop_comparison} and Perron's method.

\begin{Corollary}\label{cor_existepsilon}
Let $\Omega$ be a bounded domain which satisfies a uniform exterior sphere condition. Let Assumptions A\ref{assump_Felliptic} and A\ref{assump_theta} hold and let $f\in C(\overline{\Omega}),g\in C(\partial\Omega)$. Then, for every $0<\varepsilon<1$ and $v\in  C(\overline\Omega)$, there exists a unique viscosity solution $u_{\varepsilon}^v$ to \eqref{eq_deg1} such that $\underline{w}\leq u_{\varepsilon}^v\leq\overline{w}$ in $\overline{\Omega}$. Moreover, there exists $\beta=
\beta(d,\lambda,\Lambda)>0$ such that
for every $\Omega'\Subset \Omega$, 
\begin{equation}\label{eq:cbeta1}
	\left\|u_{\varepsilon}^v\right\|_{C^\beta(\Omega')}\,\leq\,C,
\end{equation}
for some $C=C\left(d,\lambda,\Lambda,\|u_{\varepsilon}^v\|_{L^\infty(\Omega)},\|f\|_{L^\infty(\Omega)},{\rm dist}(\Omega',\partial\Omega)\right)$.
\end{Corollary}
\begin{proof}
We only need to show \eqref{eq:cbeta1}. To this end we notice that $u_{\varepsilon}^v$ is a viscosity subsolution of
\[
	F(D^2u_{\varepsilon}^v)=\|f\|_{L^\infty(\Omega)}+\|u_{\varepsilon}^v\|_{L^\infty(\Omega)}\hspace{.4in}\mbox{in}\hspace{.1in}\{|Du|>1\},
\]
and thus it is a viscosity subsolution to
\[
	\mathcal{P}^-_{\lambda,\Lambda}(D^2u_{\varepsilon}^v)=\|f\|_{L^\infty(\Omega)}+\|u_{\varepsilon}^v\|_{L^\infty(\Omega)}\hspace{.4in}\mbox{in}\hspace{.1in}\{|Du|>1\}.	
\]
Similarly we obtain that in the set $\{|Du|>1\}$, $u_{\varepsilon}^v$ is a viscosity supersolution to
\[
	\mathcal{P}^+_{\lambda,\Lambda}(D^2u_{\varepsilon}^v)=-\|f\|_{L^\infty(\Omega)}-\|u_{\varepsilon}^v\|_{L^\infty(\Omega)}\hspace{.4in}\mbox{in}\hspace{.1in}\{|Du|>1\}.	
\]
The result now follows by an easy application of Proposition \ref{prop_mooney}.
\end{proof}

Now we introduce the set $B\subset C(\overline\Omega)$, given by
\begin{equation}\label{eq_setB}
	B:=\left\lbrace w\in C(\overline\Omega)\;\;|\;\;\underline w\leq w\leq\overline w\right\rbrace,
\end{equation}
where $\underline w,\overline w:\Omega\to\mathbb{R}$ are the sub and supersolution from Lemma \ref{lem_existepsilon}, respectively. It is clear that $B$ is a convex and closed subset of $C(\overline\Omega)$. Define a map $T:B\to C(\overline\Omega)$ as follows. Given $v\in B$, let $u_\varepsilon^v$ be the unique solution to \eqref{eq_deg1} such that $u_\varepsilon^v=g$ on $\partial\Omega$, whose existence is the subject of Corollary \ref{cor_existepsilon}. Set 
\begin{equation}\label{eq_mapT}
	Tv:=u_\varepsilon^v
\end{equation}
The next lemma collects some properties of the map $T$.

\begin{Lemma}[Properties of the map $T$]\label{lem_propT}
Let $B\subset C(\overline\Omega)$ and $T:B\to C(\overline\Omega)$ be defined as in \eqref{eq_setB} and \eqref{eq_mapT}, respectively. Then $T(B)\subset B$. In addition, $T(B)$ is precompact in $B$ and the map $T$ is continuous.
\end{Lemma}
\begin{proof}
Let $v\in B$. Corollary \ref{cor_existepsilon} and the definition of $T$ imply that $\underline w\leq Tv\leq \overline w$, hence $T(B)\subset B$. We emphasize that 
$\underline w$ and $\overline w$ are independent of $v\in C(\overline\Omega)$ and $\varepsilon>0$.

Now we observe that $T(B)$ is precompact. Let $(Tv_n)_{n\in\mathbb{N}}$ be a sequence in $T(B)$. Estimate \eqref{eq:cbeta1}, together with
$\underline w\leq Tv_n\leq \overline w$, implies that the sequence $(Tv_n)_{n\in\mathbb{N}}$ is equibounded and equicontinuous in $C(\overline\Omega)$. Hence, it has a subsequence which converges to some $w\in B$.

To complete the proof, we show that $T$ is continuous. Suppose that $(v_n)_{n\in\mathbb{N}}$ is a sequence in $B$ which converges in $C(\overline\Omega)$ to $v\in B$. We need to verify that $Tv_n\to Tv$ in $C(\overline\Omega)$. Since $T(B)$ is precompact, there exists $w\in B$ such that $Tv_n\to w$ in $C(\overline\Omega)$, through a subsequence if necessary.

We notice that $h_\varepsilon^{v_n}$ converges uniformly to $h_\varepsilon^{v}$ in $\overline\Omega$, since $v_n\to v$ uniformly. As a consequence, the sequence of operators $(G_\varepsilon^n)_{n\in\mathbb{N}}$ given by
\[
	G_\varepsilon^n(x,r,p,M):=\left(\varepsilon+|p|\right)^{\theta_\varepsilon^{v_n}(x)}\left(\varepsilon r+F(M)\right)
\]
converges locally uniformly to $G_\varepsilon^\infty$, where
\[
	G_\varepsilon^\infty(x,r,p,M):=\left(\varepsilon+|p|\right)^{\theta_\varepsilon^{v}(x)}\left(\varepsilon r+F(M)\right).
\]
The stability of viscosity solutions and the uniqueness of viscosity solutions to our Dirichlet boundary value problem for \eqref{eq_deg1} ensure that $w=Tv$. To complete the proof it remains to notice that this argument does not depend on the subsequence. 

Suppose through a different subsequence $(Tv_{n_j})_{j\in\mathbb{N}}$, we obtain $Tv_{n_j}\to w'$. Once again, the stability of viscosity solutions yields $Tv=w'$. The uniqueness of viscosity solutions ensures $w=w'$ and the proof is complete.
\end{proof}

In the sequel we detail the proof of Theorem \ref{thm_existence}.

\begin{proof}[Proof of Theorem \ref{thm_existence}]
Lemma \ref{lem_propT} and properties of the set $B$ allow us to apply the Schauder Fixed Point Theorem; see, for example, \cite[Corollary 11.2]{GT2001}. For every 
$\varepsilon>0$, we conclude that there exists a viscosity solution $u_\varepsilon\in C(\overline\Omega)$ to 
\[
	\left(\varepsilon+|Du_\varepsilon|\right)^{\theta_\varepsilon^{u_\varepsilon}(x)}\left(\varepsilon u_\varepsilon+F(D^2u_\varepsilon)\right)=f(x)\hspace{.2in}\mbox{in}\,\,\Omega,
\]
such that $u_\varepsilon=g$ on $\partial\Omega$. Again, estimate \eqref{eq:cbeta1}, together with
$\underline w\leq u_\varepsilon\leq \overline w$, ensures the existence of a sequence $(u_{\varepsilon_n})_{n\in\mathbb{N}}$, with $\varepsilon_n<1/n$, and a function $u\in B$, such that $u_{\varepsilon_n}\to u$ in $C(\overline\Omega)$. Using the fact that $\theta_\varepsilon^{u_\varepsilon}$ converges to $\theta_1\chi_{\{u>0\}}+\theta_2\chi_{\{u<0\}}$ uniformly on compact subsets of $(\Omega^+(u)\cup\Omega^-(u))\cap\Omega$, a standard consistency argument now allows us to conclude that $u$ is a viscosity solution to 
\eqref{eq_deg0} in $(\Omega^+(u)\cup\Omega^-(u))\cap\Omega$. Moreover, since $0\leq \theta_\varepsilon^{u_\varepsilon}\leq\theta_2$, $u$ is also a viscosity subsolution to
\eqref{eq_cviscsub}
and a viscosity supersolution to
\eqref{eq_cviscsuper} in $\Omega$.
\end{proof}

\section{Towards improved regularity}\label{sec_improvedreg}

In this section we prove Theorem \ref{teo_main2}. 
We first establish H\"older continuity of viscosity solutions to differential inequalities \eqref{eq_cil1}-\eqref{eq_cil2}, for arbitrary vector $q\in\mathbb{R}^d$.
Proposition \ref{prop_compactness} below is a version of Lemma 3 of \cite{Imbert-Silvestre2013} and its proof follows the strategy of the proof of 
Lemma 3 of \cite{Imbert-Silvestre2013}. However we present the proof with all details. We emphasize that the H\"older-estimate 
in Proposition \ref{prop_compactness} does not depend on $q$.

\begin{Proposition}[H\"older-continuity]\label{prop_compactness}
Let Assumptions A\ref{assump_Felliptic}, A\ref{assump_theta} hold. Let $u \in  C(B_1)$ be a viscosity subsolution to 
\eqref{eq_cil1}
and a viscosity supersolution to
\eqref{eq_cil2}
in the unit ball $B_1$, where $q\in\mathbb{R}^d$ is arbitrary. Suppose that $\|u\|_{L^\infty(B_1)}\leq 1$.
Then $u\in C^\beta_{\rm loc}(B_1)$, where $\beta$ is from Corollary \ref{prop_hrm}
and for every $0<\tau<1$, there exists a universal constant $C_\tau>0$ such that
\begin{equation}\label{eq:esttau}
	\left\|u\right\|_{ C^\beta(B_{\tau})}\,\leq\,C_\tau.
\end{equation}
\end{Proposition}

\begin{proof}
Fix $0<r<\frac{1-\tau}{2}$ and define
\[
\omega(t)\,:=\,t\,-\,\frac{t^2}{2}.
\]

For constants $L_1,\,L_2>0$ and $x_0\in B_{\tau}$, we set
\begin{equation*}
L\,:=\,\sup_{x,y\in B_r(x_0)}\left[u(x)\,-\,u(y)\,-\,L_1\omega(|x-y|)\,-\,L_2(|x-x_0|^2+|y-x_0|^2)\right].
\end{equation*}
We aim at verifying that there are choices of $L_1$ and $L_2$ for which $L\leq 0$ for every $x_0\in B_{\tau}$. This will imply that
$u$ is Lipschitz continuous in $B_{\tau}$ by taking $x_0=x$. 

We argue by contradiction. Suppose there exists $x_0\in B_{\tau}$ for which $L>0$ regardless of the choices of $L_1$ and $L_2$. Consider the auxiliary functions $\psi,\,\phi:\overline B_1\times \overline B_1\to\mathbb{R}$ given by
\[
\psi(x,y)\,:=\, L_1\omega(\left|x\,-\,y\right|)\,+\,L_2\left(\left|x\,-\,x_0\right|^2\,+\,\left|y\,-\,x_0\right|^2\right)
\]
and
\[
	\phi(x,y)\,:=\,u(x)\,-\,u(y)\,-\,\psi(x,y).
\]
Let $(\overline{x},\overline{y})$ be a point where $\phi$ attains its maximum. Then
\[
	\phi(\overline{x},\overline{y})\,=\,L\,>\,0
\]
and
 \[
L_1\omega(\left|\overline{x}\,-\,\overline{y}\right|)\,+\,L_2\left(\left|\overline{x}\,-\,x_0\right|^2\,+\,\left|\overline{y}\,-\,x_0\right|^2\right)\,\leq\,\sup_{x\in B_1}u(x)\,-\,\inf_{x\in B_1}u(x)\,\leq\,2.
\]
Set
\[
	L_2\,:=\,\left(\frac{4\sqrt{2}}{r}\right)^2.
\]
Then, 
\[
	\left|\overline{x}\,-\,x_0\right|\,+\,\left|\overline{y}\,-\,x_0\right|\,\leq\,\frac{r}{2}.
\]
It follows that $\overline{x},\overline{y}\in B_r(x_0)$. In addition, $\overline{x}\neq\overline{y}$; indeed, if this is not the case, we would conclude $L\leq 0$.

At this point, we use Proposition \ref{prop_cil} to obtain elements in the closures of subjets and superjets of $u$ and produce a viscosity inequality relating those elements. We split the rest of the proof into four steps.

\medskip

\noindent{\bf Step 1 - }Proposition \ref{prop_cil} ensures the existence of $(q_{\overline{x}},X)$ in the closure of the subjet of $u$ at $\overline{x}$ and of $(q_{\overline{y}},Y)$ in the closure of the superjet of $u$ at $\overline{y}$, with 
\[
 q_{\bar x}:= D_x\psi(\bar x,\bar y)=L_1\omega'(|\bar x-\bar y|)\frac{\bar x-\bar y}{|\bar x-\bar y|}+2L_2(\bar x-x_0)
\]
and
\[
q_{\bar y}:= -D_y\psi(\bar x,\bar y)=L_1\omega'(|\bar x-\bar y|)\frac{\bar x-\bar y}{|\bar x-\bar y|}-2L_2(\bar y-x_0).
\]
In addition, the matrices $X$ and $Y$ satisfy the inequality
\begin{equation}\label{matrix_inequality1}
	 \left(
		 \begin{array}{ccc}
			 X   & 0 \\
			 0  &-Y 
		 \end{array}
	 \right)
	 \leq 
	 \left(
		 \begin{array}{ccc}
	 		Z  & -Z \\
			 -Z  & Z
		 \end{array}
	 \right)
	 			+(2L_2+\iota)I,
\end{equation}
for 
\[
Z:= L_1\omega''(|\bar x-\bar y|)\frac{(\bar x-\bar y)\otimes(\bar x-\bar y)}{|\bar x-\bar y|^2}+L_1\frac{\omega'(|\bar x-\bar y|)}{|\bar x-\bar y|}\left(I-\frac{(\bar x-\bar y)\otimes(\bar x-\bar y)}{|\bar x-\bar y|^2}\right)
\]
where $0<\iota\ll 1$ depends solely on the norm of $Z$. 
	 
Next, we apply the matrix inequality \eqref{matrix_inequality1} to special vectors as to recover information about the eigenvalues of $X-Y$. First, apply \eqref{matrix_inequality1} to vectors of the form $(z,z)\in\mathbb{R}^{2d}$ to get
\[
	 \langle(X\,-\,Y)z,z\rangle\,\leq\,(4L_2\,+\,2\iota)|z|^2.
\]
We then conclude that all the eigenvalues of $(X-Y)$ are less than or equal to $4L_2+2\iota.$ Now, by applying \eqref{matrix_inequality1} to 
\[
	 \bar{z}:=\left(\frac{\bar x-\bar y}{|\bar x-\bar y|},\frac{\bar y-\bar x}{|\bar x-\bar y|}\right),
\]
we obtain
\[
\left\langle(X-Y)\frac{\bar x-\bar y}{|\bar x-\bar y|},\frac{\bar x-\bar y}{|\bar x-\bar y|}\right\rangle\leq(4L_2+2\iota)\left|\frac{\bar x-\bar y}{|\bar x-\bar y|}\right|^2+4L_1\omega''(|\bar x-\bar y|)=4L_2+2\iota-4L_1.
\]
We conclude that at least one eigenvalue of $(X-Y)$ is below 
\[
4L_2+2\iota-4L_1.
\]
We notice this quantity will be negative for large values of $L_1$. Evaluating the minimal Pucci operator on $X-Y$, we then get
\begin{equation}\label{eq:puccimin}
\mathcal{P}^{-}_{\lambda,\Lambda}(X-Y)\geq 4\lambda L_1-(\lambda+(d-1)\Lambda)(4L_2+2\iota).
\end{equation}	 
At this point we evoke the differential inequalities \eqref{eq_cil1} and \eqref{eq_cil2} satisfied by $u$ in the viscosity sense. They yield
\begin{equation}\label{eq_cil3}
	 \min\left\lbrace |q+q_{\bar x}|^{\theta_2}F(X),\,F(X)\right\rbrace\,\leq\,1
\end{equation}
and 
\begin{equation}\label{eq_cil4}
	\max\left\lbrace |q+q_{\bar y}|^{\theta_2}F(Y),\,F(Y)\right\rbrace\,\geq\,-1.
\end{equation}
Since $F$ is $(\lambda,\Lambda)$-elliptic, we have
\begin{equation}\label{eq:Fell}
	 F(X)\geq F(Y)+\mathcal{P}^{-}_{\lambda,\Lambda}(X-Y).
\end{equation}

\medskip

\noindent{\bf Step 2 - }In what follows, we relate \eqref{eq:puccimin}, \eqref{eq:Fell} and \eqref{eq_cil3}-\eqref{eq_cil4}. Below, we consider all possible cases.

\medskip

\noindent\emph{Case 1:} Suppose 
\[
	\min\left\lbrace |q+q_{\bar x}|^{\theta_2}F(X),\,F(X)\right\rbrace\,=\,F(X)
\]
and
\[
	\max\left\lbrace |q+q_{\bar x}|^{\theta_2}F(Y),\,F(Y)\right\rbrace\,=\,F(Y).
\]
In this case we get
\[
4\lambda L_1\leq (\lambda+(d-1)\Lambda)(4L_2+2\iota)+2.
\]

\medskip

\noindent\emph{Case 2:} Suppose
\[
	\min\left\lbrace |q+q_{\bar x}|^{\theta_2}F(X),\,F(X)\right\rbrace\,=\,\left|q\,+\,q_{\overline{x}}\right|^{\theta_2}F(X)
\]
and
\[
	\max\left\lbrace |q+q_{\bar x}|^{\theta_2}F(Y),\,F(Y)\right\rbrace\,=\,\left|q\,+\,q_{\overline{y}}\right|^{\theta_2}F(Y).
\]
Then,
\[
4\lambda L_1\leq (\lambda+(d-1)\Lambda)(4L_2+2\iota)\,+\,\left|q\,+\,q_{\overline{y}}\right|^{-\theta_2}\,+\,\left|q\,+\,q_{\overline{x}}\right|^{-\theta_2}.
\]

\medskip

\noindent\emph{Case 3:} Suppose
\[
	\min\left\lbrace |q+q_{\bar x}|^{\theta_2}F(X),\,F(X)\right\rbrace\,=\,F(X)
\]
and
\[
	\max\left\lbrace |q+q_{\bar x}|^{\theta_2}F(Y),\,F(Y)\right\rbrace\,=\,\left|q\,+\,q_{\overline{y}}\right|^{\theta_2}F(Y).
\]
In this case we produce
\[
4\lambda L_1\leq (\lambda+(d-1)\Lambda)(4L_2+2\iota)\,+\,\left|q\,+\,q_{\overline{y}}\right|^{-\theta_2}\,+\,1.
\]

\medskip

\noindent\emph{Case 4:} Suppose
\[
	\min\left\lbrace |q+q_{\bar x}|^{\theta_2}F(X),\,F(X)\right\rbrace\,=\,\left|q\,+\,q_{\overline{x}}\right|^{\theta_2}F(X)
\]
and
\[
	\max\left\lbrace |q+q_{\bar x}|^{\theta_2}F(Y),\,F(Y)\right\rbrace\,=\,F(Y).
\]
Here, we obtain
\[
4\lambda L_1\leq (\lambda+(d-1)\Lambda)(4L_2+2\iota)\,+\,1\,+\,\left|q\,+\,q_{\overline{x}}\right|^{-\theta_2}.
\]

%
\medskip

\noindent{\bf Step 3 - }Next, we explore the fact that $q\in\mathbb{R}^d$ is arbitrary, in close connection with Cases 1-4. Observe that
\[
|q_{\bar x}|\leq L_1(1+|\bar x-\bar y|)+2L_2\leq aL_1,
\]
for some $a>0$, universal. Let  $A_0:= 10aL_1$ and assume $|q|\geq A_0$. In this case,  we ensure that $q\neq q_{\bar x}$. A similar reasoning leads to $q\neq q_{\bar y}$.
	 
From the choice of $A_0$ we conclude that
\[
	|q+q_{\bar x}|\geq A_0-\frac{A_0}{10}=\frac{9}{10}A_0;
\]
also $|q+q_{\bar y}|\geq \frac{9}{10}A_0$. Hence, we get
\[
	|q+q_{\bar x}|^{-\theta_2}\leq \left(\frac{9}{10}A_0\right)^{-\theta_2}.
\]
Similarly, we obtain
\[
	|q+q_{\bar y}|^{-\theta_2}\leq \left(\frac{9}{10}A_0\right)^{-\theta_2}.
\]

Thus, the choice of $A_0$ ensures that in all Cases 1-4
\[
4\lambda L_1\leq (\lambda+(d-1)\Lambda)(4L_2+2\iota)\,+\,\frac{C}{L_1^{\theta_2}}+C.
\]
By choosing $L_1\gg 1$ large enough, depending only on $d,\,\lambda,\,\Lambda,\,\theta_2$ and $L_2$, which in turn depends only on $0<r\ll 1$
and $\tau$, we produce a contradiction. Therefore, we cannot have $L>0$.

As a result, in the case $|q|\geq A_0$, solutions to \eqref{eq_cil1}-\eqref{eq_cil2} are locally Lipschitz-continuous, with universal estimates.

\medskip

\noindent{\bf Step 4 - } If $|q|<A_0$ then Corollary \ref{prop_hrm}, applied with $\gamma=2A_0$, guarantees that $u\in C^{\beta}_{\rm loc}(B_1)$ for some universal $\beta\in(0,1)$ and $u$ satisfies \eqref{eq:esttau}.
\end{proof}

\begin{Proposition}[Approximation Lemma]\label{prop_approx}
Let Assumptions A\ref{assump_Felliptic}, A\ref{assump_theta} hold. Let $u \in  C(B_1)$ be a viscosity subsolution to 
\begin{equation}\label{eq_compact111}
\min\left\lbrace |q+Du|^{\theta_2}F(D^2u),\,F(D^2u)\right\rbrace\,=\,c
\end{equation}
and a viscosity supersolution to
\begin{equation}\label{eq_compact211}
\max\left\lbrace |q+Du|^{\theta_2}F(D^2u),\,F(D^2u)\right\rbrace\,=\,-c
\end{equation}
in the unit ball $B_1$, where $c>0$ and $q\in\mathbb{R}^d$ is arbitrary. 
For every $0<\delta<1$ there exists $0<\varepsilon<1$ such that, if $\|u\|_{L^\infty(B_1)}\leq 1$ and
\[
	c\,\leq\,\varepsilon,
\]
then one can find $h\in C^{1,\alpha_0}_{\rm loc}(B_1)$ satisfying $\overline F(D^2h)=0$ in the viscosity sense in $B_1$ for some
$\overline F$ satisfying Assumption A\ref{assump_Felliptic}, such that
\begin{equation}\label{eq:deltaeasu}
	\left\|u\,-\,h\right\|_{L^\infty(B_{3/4})}\,\leq\,\delta.
\end{equation}
Such function $h$ satisfies
\begin{equation}\label{eq:esth1}
	\left\|h\right\|_{ C^{1,\alpha_0}(B_{1/2})}\leq C\left\|h\right\|_{L^\infty(B_{3/4})},
\end{equation}
where $C=C(d,\lambda,\Lambda)>0$ is from Proposition \ref{th:alpha0} and is independent of $q$.
\end{Proposition}
\begin{proof}For ease of presentation, we split the proof into six steps. As before, we resort to a contradiction argument.

\medskip

\noindent{\bf Step 1 - }Suppose the statement of the proposition is false. If this is the case, there exist $0<\delta_0<1$, a sequence of functions 
$(u_n)_{n\in\mathbb{N}}\subset C(B_1)$, $\|u_n\|_{L^\infty(B_1)}\leq 1$, a sequence of positive numbers $(c_n)_{n\in\mathbb{N}}$, 
$c_n\to 0$, a sequence of vectors $(q_n)_{n\in\mathbb{N}}$ and a sequence $(F_n)_{n\in\mathbb{N}}$ of operators satisfying
A\ref{assump_Felliptic}, such that 
$u_n$ is a viscosity subsolution to 
\begin{equation}\label{eq_compact1}
\min\left\lbrace |q_n+Du_n|^{\theta_2}F_n(D^2u_n),\,F_n(D^2u_n)\right\rbrace\,=\,c_n
\end{equation}
and a viscosity supersolution to
\begin{equation}\label{eq_compact2}
\max\left\lbrace |q_n+Du_n|^{\theta_2}F_n(D^2u_n),\,F_n(D^2u_n)\right\rbrace\,=\,-c_n
\end{equation}
in the unit ball $B_1$ for every $n\in\mathbb{N}$ but 
\[
	\left\|u_n\,-\,h\right\|_{L^\infty(B_{3/4})}\,\geq\,\delta_0,
\]
for every $h\in C^{1,\alpha_0}_{\rm loc}(B_{1})$ satisfying $\overline F(D^2h)=0$ in the viscosity sense in $B_1$ for some $\overline F$.

\medskip

\noindent{\bf Step 2 - } By Proposition \ref{prop_compactness}, we know that $(u_n)_{n\in\mathbb{N}}$ is bounded in $ C^\beta(B_{\tau})$ for every $0<\tau<1$. Therefore, choosing a subsequence if necessary, it converges uniformly on every compact subset of $B_1$ to some function $u_\infty\in  C^\beta_{\rm loc}(B_1)$, where $\|u_\infty\|_{L^\infty(B_1)}\leq 1$. In addition, $(F_n)_{n\in\mathbb{N}}$ is uniformly Lipschitz continuous. Hence, there exists an operator $F_\infty$ satisfying Assumption A\ref{assump_Felliptic} such that $F_n$ converges to $F_\infty$, locally uniformly. 

Our goal is to prove that $u_\infty$ is a viscosity solution to
\[
	F_\infty(D^2u_\infty)\,=\,0\hspace{.4in}\mbox{in}\hspace{.1in}B_{1}.
\]
We will only show that $u_\infty$ is a viscosity subsolution of $F_\infty(D^2u_\infty)= 0$ as the proof of the supersolution property is analogous.
We consider two cases, depending on the behavior of the sequence $(q_n)_{n\in\mathbb{N}}$. 

\medskip

\noindent{\bf Step 3 - }Firstly, suppose that the sequence $(q_n)_{n\in\mathbb{N}}$ does not admit a convergent subsequence. Then, $|q_n|\to\infty$ as $n\to\infty$. Let $\varphi\in C^2(B_1)$ and suppose that $u_\infty-\varphi$ attains a local maximum at $x_0\in B_1$; we assume this maximum to be strict. Suppose that
\[
	F_\infty(D^2\varphi(x_0))\,>\,0.
\]
Standard arguments yield a sequence $(x_n)_{n\in\mathbb{N}}$, converging to $x_0$, such that $u_n-\varphi$ attains a local maximum at $x_n$. Notice also that $D\varphi(x_n)\to D\varphi(x_0)$ and $D^2\varphi(x_n)\to D^2\varphi(x_0)$. We choose $M\in\mathbb{N}$ such that 
\[
	F_n(D^2\varphi(x_n))\,>\,c_n
\]
and
\[
	\left|q_n\,+\,D\varphi(x_n)\right|\,>\,1,
\]
for $n>M$. Then, for $n>M$, we conclude 
\[
\min\left\lbrace |q_n+D\varphi(x_n)|^{\theta_2}F_n(D^2\varphi(x_n)),F_n(D^2\varphi(x_n))\right\rbrace>c_n
\]
which is a contradiction. A similar argument, using \eqref{eq_compact2}, shows that $u_\infty$ is a viscosity supersolution and so we get
that $u_\infty$ is a viscosity solution to
\[
	F_\infty(D^2u_\infty)\,=\, 0\hspace{.4in}\mbox{in}\hspace{.1in}B_{1}
\]
in this case. 

\medskip

\noindent{\bf Step 4 - }We now consider the complementary case. Namely, suppose $(q_n)_{n\in\mathbb{N}}$ admits a convergent subsequence, still denoted by $(q_n)_{n\in\mathbb{N}}$, such that $q_n\to q_\infty$. By resorting to standard stability results (see, for instance, \cite[Section 6, Remarks 6.2 and 6.3]{Crandall-Ishii-Lions1992}), we conclude that $u_\infty$ is a viscosity subsolution to 
\begin{equation}\label{eq_contradiction2}
	\min\left\lbrace |q_\infty+Du_\infty|^{\theta_2}F_\infty(D^2u_\infty),F_\infty(D^2u_\infty)\right\rbrace =0
\end{equation}
and a viscosity supersolution to
\begin{equation}\label{eq_contradiction3}
	\max\left\lbrace |q_\infty+Du_\infty|^{\theta_2}F_\infty(D^2u_\infty),F_\infty(D^2u_\infty)\right\rbrace =0
\end{equation}
in the unit ball $B_1$. 

We now reduce the problem to the case $q_\infty\equiv 0$. Indeed, by considering $w_\infty:=u_\infty+q_\infty\cdot x$ we get that $w_\infty$ is a viscosity subsolution to 
\begin{equation}\label{eq_contradiction21}
	\min\left\lbrace |Dw_\infty|^{\theta_2}F_\infty(D^2w_\infty),F_\infty(D^2w_\infty)\right\rbrace=0
\end{equation}
and a viscosity supersolution to
\begin{equation}\label{eq_contradiction31}
	\max\left\lbrace |Dw_\infty|^{\theta_2}F_\infty(D^2w_\infty),F_\infty(D^2w_\infty)\right\rbrace=0
\end{equation}
in $B_1$. Because $D^2w_\infty=D^2u_\infty$ in the viscosity sense, by verifying $F_\infty(D^2w_\infty)=0$, we infer $F_\infty(D^2u_\infty)=0$.
We will only argue that $w_\infty$ is a viscosity subsolution of $F_\infty(D^2w_\infty)=0$. 

Let $p(x)$ be a second order polynomial of the form
\[
	p(x)\,:=\,b\cdot x\,+\,\frac{1}{2}x^TAx,
\]
for a vector $b\in\mathbb{R}^d$ and a matrix $A\in\mathbb{R}^{d^2}$, fixed. Suppose that $w_\infty-p$ attains its maximum at $x_0\in B_1$. Without loss of generality we suppose $x_0=0$, $w_\infty(0)=0$ and the maximum is strict. Our goal is to prove that $F_\infty(A)\leq 0$. 

From \eqref{eq_contradiction21} we infer one of the following inequalities:
\[
	|b|^{\theta_1}F_\infty(A)\leq 0, 	\hspace{.2in}|b|^{\theta_2}F_\infty(A)\leq 0\hspace{.15in}\mbox{or}\hspace{.15in}F_\infty(A)\leq 0.
\]
In case $b\neq 0$, $F_\infty(A)\leq 0$ and we are done. If $b=0$ the argument is exactly the same as that in the proof of \cite[Lemma 6]{Imbert-Silvestre2013}.

\medskip

\noindent{\bf Step 5 - } Since $u_\infty$ is a viscosity solution to
$F_\infty(D^2u_\infty)=0$ in $B_{1}$, which satisfies Assumption A\ref{assump_Felliptic}, Proposition \ref{th:alpha0} guarantees that $u_\infty\in C^{1,\alpha_0}_{\rm loc}(B_1)$.
Proposition \ref{th:alpha0} implies
\[
	\left\|u_\infty\right\|_{ C^{1,\alpha_0}(B_{1/2})}\,\leq\,C\left\|u_\infty\right\|_{L^\infty(B_{3/4})}.
\]
Thus, taking $h=u_\infty,\overline F=F_\infty$, we reach a contradiction. This completes the proof.

\end{proof}


\begin{Proposition}[Oscillation control]\label{prop_osc1}
Let Assumptions A\ref{assump_Felliptic}, A\ref{assump_theta} hold. Let $u\in  C(B_1)$ and $\|u\|_{L^\infty(B_{1})}\leq 1$. Let
$\alpha\,\in\,(0,\alpha_0)$.
Set
\[
\rho\,:=\,\min\left\{\frac{1}{2},\left(\frac{1}{2C}\right)^\frac{1}{\alpha_0-\alpha}\right\}\hspace{.4in}\mbox{and}\hspace{.4in}
\delta\,:=\,\frac{\rho^{1+\alpha}}{2},
\]
where $C$ is from Proposition \ref{prop_approx}. If $u$ is a viscosity subsolution to \eqref{eq_compact111} and a viscosity supersolution to \eqref{eq_compact211} for any $q$ and for $c= \varepsilon$, where $\varepsilon$ is from 
Proposition \ref{prop_approx} for the $\delta$ above,
then there exists an affine function $\ell(x)=a+b\cdot x$ such that $|a|\leq C,|b|\leq C$ and
\[
	\sup_{x\in B_\rho}\,\left|u(x)\,-\,\ell(x)\right|\,\leq\,\rho^{1+\alpha}.
\]
\end{Proposition}

\begin{proof}
Let $h$ be the approximating function whose existence is ensured by Proposition \ref{prop_approx}. Set $\ell(x):=h(0)+Dh(0)\cdot x$. The triangle inequality yields
\[
	\begin{split}
		\sup_{x\in B_\rho}\left|u(x)-\ell(x)\right|&\leq\sup_{x\in B_\rho}\left|u(x)-h(x)\right|+\sup_{x\in B_\rho}\left|h(x)-h(0)-Dh(0)\cdot x\right|\\
			&\leq\,\delta\,+\,C\rho^{1\,+\,\alpha_0}\leq \rho^{1+\alpha}.
	\end{split}
\]
\end{proof}

\begin{Proposition}[Oscillation control in discrete scales]\label{prop_oscdiscrete}
Let the assumptions of Proposition \ref{prop_osc1} be satisfied, however let $q=0$ now and let in addition
\[
	\alpha\leq\frac{1}{1\,+\,\theta_2}.
\]
There exists a sequence of affine functions $(\ell_n)_{n\in\mathbb{N}}$, of the form
\[
	\ell_n(x)\,:=\,a_n\,+\,b_n\cdot x,
\]
satisfying
\begin{equation}\label{eq_c1alpha1}
	\sup_{x\in B_{\rho^n}}\,\left|u(x)\,-\,\ell_n(x)\right|\,\leq\,\rho^{(1+\alpha)n}
\end{equation}
and
\begin{equation}\label{eq_c1alpha2}
	\left|a_{n+1}\,-\,a_n\right|\,+\,\rho^{n}\left|b_{n+1}\,-\,b_n\right|\,\leq\,C\rho^{(1+\alpha)n},
\end{equation}
for every $n\in\mathbb{N}$ and some universal constant $C>0$.
\end{Proposition}

\begin{proof}
We argue by induction.

\medskip

\noindent{\bf Step 1 - } The basis case follows from Proposition \ref{prop_osc1}. In fact, for $h$ as in that proposition, set $\ell_0\equiv 0$ and define $\ell_1(x):=h(0)+Dh(0)\cdot x$. It is clear that 
\[
	\sup_{x\in B_\rho}\left|u(x)\,-\,\ell_1(x)\right|\,\leq\,\rho^{1+\alpha}
\]
and
\[
	|h(0)|\,+\,|Dh(0)|\,\leq\,C,
\]
where $C>0$ is a universal constant.

\medskip

\noindent{\bf Step 2 - } Now, suppose the case $n=k$ has been verified. Next we examine the case $n=k+1$. To that end, we introduce the auxiliary function
\[
	v_k(x)\,:=\,\frac{u(\rho^kx)\,-\,\ell_k(\rho^kx)}{\rho^{k(1+\alpha)}}.
\]
It is easy to see that (see Section \ref{subsec_smallness}) $v_k$ is a viscosity subsolution to 
\[
	\min\left\lbrace \left|\rho^{-k\alpha}b_k+Dv_k\right|^{\theta_2}F_k(D^2v_k),F_k(D^2v_k) \right\rbrace=c_k
\]
and a viscosity supersolution to
\[
	\max\left\lbrace \left|\rho^{-k\alpha}b_k+Dv_k\right|^{\theta_2}F_k(D^2v_k),F_k(D^2v_k) \right\rbrace=-c_k,
\]
where
\[
	F_k(M)\,:=\,\rho^{k(1-\alpha)}F\left(\frac{1}{\rho^{k(1-\alpha)}}M\right)
\]
and
\[
	c_k\,:=\,\varepsilon\left(\max\left\{\rho^{k(1-\alpha(1+\theta_2))},\rho^{k(1-\alpha)}\right\}\right).
\]
Notice that $F_k$ satisfies Assumption A\ref{assump_Felliptic} and the choice of the exponent 
$\alpha$, together with $\rho\leq 1/2$, yields $c_k\leq\varepsilon$.
The induction hypothesis ensures $\|v_k\|_{L^\infty(B_1)}\leq 1$. 

\medskip

\noindent{\bf Step 3 - }Proposition \ref{prop_osc1} now implies that there exists an affine function $\overline{\ell}_k$ satisfying
\begin{equation}\label{eq_vk}
		\sup_{x\in B_\rho}\left|v_k(x)\,-\,\overline{\ell}_k(x)\right|\,\leq\,\rho^{1+\alpha},
\end{equation}
where
\[
	\overline{\ell}_k(x)\,:=\,\overline{a}_k\,+\,\overline{b}_k\cdot x
\]
is such that $|\overline{a}_k|$ and $|\overline{b}_k|$ are bounded by a constant $C>0$, depending solely on $d,\,\lambda$ and $\Lambda$. Using the definition of $v_k$, estimate \eqref{eq_vk} becomes
\[
	\begin{split}
		\sup_{x\in B_\rho}\left|\frac{u(\rho^kx)\,-\,a_k\,-\,b_k\cdot(\rho^kx)\,-\,\rho^{k(1+\alpha)}\left(\overline{a}_k\,+\,\overline{b}_k\cdot x\right)}{\rho^{k(1+\alpha)}}\right|\,\leq\,\rho^{1+\alpha}.
	\end{split}
\]
It implies
\[
	\sup_{x\in B_{\rho^{k+1}}}\left|u(x)\,-\,\ell_{k+1}(c)\right|\,\leq\,\rho^{(k+1)(1+\alpha)}
\]
for
\[
	\ell_{k+1}\,=\,a_{k+1}\,+\,b_{k+1}\cdot x\,:=\,\left(a_k\,+\,\rho^{k(1+\alpha)}\overline{a}_k\right)\,+\,\left(b_k\,+\,\rho^{k\alpha}\overline{b}_k\right)\cdot x.
\]
To complete the proof, we notice that 
\[
	\left|a_{k+1}\,-\,a_k\right|\,\leq\,\left|\overline{a}_k\right|\rho^{k(1+\alpha)}\,\leq\,C\rho^{k(1+\alpha)}
\]
and
\[
	\rho^k\left|b_{k+1}\,-\,b_k\right|\,\leq\,\left|\overline{b}_k\right|\rho^{k(1+\alpha)}\,\leq\,C\rho^{k(1+\alpha)}.
\]

\end{proof}

\begin{proof}[Proof of Theorem \ref{teo_main2}]
Let $\delta,\varepsilon$ be as in Proposition \ref{prop_osc1}. Denote 
\[
K=\|u\|_{L^\infty(B_{1})}+\max\left\lbrace\left\|f\right\|_{L^\infty(B_1)},\left\|f\right\|^\frac{1}{1+\theta_2}_{L^\infty(B_1)} \right\rbrace
\]
and set
\[
v(x)=\frac{u(\varepsilon x)}{K}.
\]
Then  $\|v\|_{L^\infty(B_{1})}\leq 1$ and 
(see Section \ref{subsec_smallness}) $v$ is a viscosity subsolution to
\[
\min\left\lbrace |Dv|^{\theta_2}\overline F(D^2v),\,\overline F(D^2v)\right\rbrace
=\varepsilon\quad\mbox{in}\hspace{.1in}B_{1}
\]
and a viscosity supersolution to
\[
\max\left\lbrace |Dv|^{\theta_2}\overline F(D^2v),\,\overline F(D^2v)\right\rbrace
=
-\varepsilon\quad\mbox{in}\hspace{.1in}B_{1},
\]
where 
\[
\overline F(X)=\frac{\varepsilon^2}{K}F\left(\frac{K}{\varepsilon^2}X\right).
\]
Thus we are in the framework of 
Propositions \ref{prop_approx}-\ref{prop_oscdiscrete}.
From \eqref{eq_c1alpha1}-\eqref{eq_c1alpha2}  we infer
\[
	a_n\,\longrightarrow\,v(0)\hspace{.4in}\mbox{and}\hspace{.4in}b_n\,\longrightarrow\,Dv(0),
\]
with convergence rates of the order
\[
	\left|a_n\,-\,v(0)\right|\,\sim\,\rho^{n(1+\alpha)}\hspace{.3in}\mbox{and}\hspace{.3in}\left|b_n\,-\,Dv(0)\right|\,\sim\,\rho^{n\alpha}.
\]
Therefore
\[
	\begin{split}
		\sup_{x\in B_{\rho^n}}\,\left|v(x)\,-\,v(0)\,-\,Dv(0)\cdot x\right|\,&\leq\,\sup_{x\in B_{\rho^n}}\,\left|v(x)\,-\,a_n\,-\,b_n\cdot x\right|\,+\,C\rho^{n(1+\alpha)}\\
			&\leq\,C\rho^{n(1+\alpha)},
	\end{split}
\]
where the second inequality comes from \eqref{eq_c1alpha1}. Now if $\rho^{m+1}<r\leq \rho^m$ for some $m\in \mathbb{N}$, then
\[
	\begin{split}
		\sup_{x\in B_r}\,\left|v(x)\,-\,v(0)\,-\,Dv(0)\cdot x\right|\,&\leq\,\sup_{x\in B_{\rho^m}}\,\left|v(x)\,-\,v(0)\,-\,Dv(0)\cdot x\right|\\
			&\leq\,C\rho^{m(1+\alpha)}\\
			&\leq\,\frac{C}{\rho^{1+\alpha}}\rho^{(m+1)(1+\alpha)}\\
			&\leq\,Cr^{1+\alpha}.
	\end{split}
\]
This concludes the proof since the same argument can be done for every point $x_0\in B_\tau$ provided we also take 
$\varepsilon<1-\tau$.
\end{proof}

\bigskip

\paragraph{Acknowledgements:} This work was partially supported by the Centre for Mathematics of the University of Coimbra - UIDB/00324/2020, funded by the Portuguese Government through FCT/MCTES. EP is partly supported by CNPq-Brazil (Grants \# 433623/2018-7 and 307500/2017-9), FAPERJ (Grant \# E-26/200.002/2018) and Instituto Serrapilheira (Grant \#1811-25904). GR is partly supported by CAPES. This study was financed in part by the Coordena\c{c}\~ao de Aperfei\c{c}oamento de Pessoal de N\'ivel Superior - Brasil (CAPES) - Finance Code 001.

\bibliography{biblio}

\begin{thebibliography}{10}

\bibitem{Amaral-Teixeira2015}
Marcelo~D. Amaral and Eduardo~V. Teixeira.
\newblock Free transmission problems.
\newblock {\em Comm. Math. Phys.}, 337(3):1465--1489, 2015.

\bibitem{Araujo-Ricarte-Teixeira2015}
Dami\~{a}o~J. Ara\'{u}jo, Gleidson Ricarte, and Eduardo~V. Teixeira.
\newblock Geometric gradient estimates for solutions to degenerate elliptic
  equations.
\newblock {\em Calc. Var. Partial Differential Equations}, 53(3-4):605--625,
  2015.

\bibitem{Bao-Li-Yin2009}
Ellen~Shiting Bao, YanYan Li, and Biao Yin.
\newblock Gradient estimates for the perfect conductivity problem.
\newblock {\em Arch. Ration. Mech. Anal.}, 193(1):195--226, 2009.

\bibitem{Bao-Li-Yin2010}
Ellen~Shiting Bao, YanYan Li, and Biao Yin.
\newblock Gradient estimates for the perfect and insulated conductivity
  problems with multiple inclusions.
\newblock {\em Comm. Partial Differential Equations}, 35(11):1982--2006, 2010.

\bibitem{Birindelli-Demengel2004}
Isabeau Birindelli and Fran\c{c}oise Demengel.
\newblock Comparison principle and {L}iouville type results for singular fully
  nonlinear operators.
\newblock {\em Ann. Fac. Sci. Toulouse Math. (6)}, 13(2):261--287, 2004.

\bibitem{Birindelli-Demengel2006}
Isabeau Birindelli and Fran\c{c}oise Demengel.
\newblock First eigenvalue and maximum principle for fully nonlinear singular
  operators.
\newblock {\em Adv. Differential Equations}, 11(1):91--119, 2006.

\bibitem{Birindelli-Demengel2007A}
Isabeau Birindelli and Fran\c{c}oise Demengel.
\newblock The {D}irichlet problem for singular fully nonlinear operators.
\newblock {\em Discrete Contin. Dyn. Syst.}, (Dynamical systems and
  differential equations. Proceedings of the 6th AIMS International Conference,
  suppl.):110--121, 2007.

\bibitem{Birindelli-Demengel2007B}
Isabeau Birindelli and Fran\c{c}oise Demengel.
\newblock Eigenvalue, maximum principle and regularity for fully nonlinear
  homogeneous operators.
\newblock {\em Commun. Pure Appl. Anal.}, 6(2):335--366, 2007.

\bibitem{Birindelli-Demengel2008}
Isabeau Birindelli and Fran\c{c}oise Demengel.
\newblock Eigenvalue and {D}irichlet problem for fully-nonlinear operators in
  non-smooth domains.
\newblock {\em J. Math. Anal. Appl.}, 352(2):822--835, 2009.

\bibitem{Birindelli-Demengel2014}
Isabeau Birindelli and Fran\c{c}oise Demengel.
\newblock {$C^{1,\beta}$} regularity for {D}irichlet problems associated to
  fully nonlinear degenerate elliptic equations.
\newblock {\em ESAIM Control Optim. Calc. Var.}, 20(4):1009--1024, 2014.

\bibitem{Bonnetier2000}
Eric Bonnetier and Michael Vogelius.
\newblock An elliptic regularity result for a composite medium with
  ``touching'' fibers of circular cross-section.
\newblock {\em SIAM J. Math. Anal.}, 31(3):651--677, 2000.

\bibitem{Borsuk1968}
Mikhail~V. Borsuk.
\newblock A priori estimates and solvability of second order quasilinear
  elliptic equations in a composite domain with nonlinear boundary condition
  and conjugacy condition.
\newblock {\em Trudy Mat. Inst. Steklov.}, 103:15--50. (loose errata), 1968.

\bibitem{Borsuk2010}
Mikhail~V. Borsuk.
\newblock {\em Transmission problems for elliptic second-order equations in
  non-smooth domains}.
\newblock Frontiers in Mathematics. Birkh\"{a}user/Springer Basel AG, Basel,
  2010.

\bibitem{Briane-Capdeboscq-Nguyen2013}
Marc Briane, Yves Capdeboscq, and Luc Nguyen.
\newblock Interior regularity estimates in high conductivity homogenization and
  application.
\newblock {\em Arch. Ration. Mech. Anal.}, 207(1):75--137, 2013.

\bibitem{Bronzi-Pimentel-Rampasso-Teixeira2018}
Anne~C. Bronzi, Edgard~A. Pimentel, Giane~C. Rampasso, and Eduardo~V. Teixeira.
\newblock Regularity of solutions to a class of variable-exponent fully
  nonlinear elliptic equations.
\newblock {\em J. Funct. Anal.}, 279(12):108781, 31, 2020.

\bibitem{Caffarelli-Cabre1995}
Luis~A. Caffarelli and Xavier Cabr\'{e}.
\newblock {\em Fully nonlinear elliptic equations}, volume~43 of {\em American
  Mathematical Society Colloquium Publications}.
\newblock American Mathematical Society, Providence, RI, 1995.

\bibitem{CCKS1996}
Luis~A. Caffarelli, Michael~G. Crandall, Maciej Kocan, and Andrzej
  \'{S}wi\k{e}ch.
\newblock On viscosity solutions of fully nonlinear equations with measurable
  ingredients.
\newblock {\em Comm. Pure Appl. Math.}, 49(4):365--397, 1996.

\bibitem{Caffarelli-Carro-Stinga2020}
Luis~A. Caffarelli, Mar\'{\i}a Soria-Carro, and Pablo~R. Stinga.
\newblock Regularity for {$C^{1,\alpha}$} interface transmission problems.
\newblock {\em Arch. Ration. Mech. Anal.}, 240(1):265--294, 2021.

\bibitem{Campanato1957}
Sergio Campanato.
\newblock Sul problema di {M}. {P}icone relativo all'equilibrio di un corpo
  elastico incastrato.
\newblock {\em Ricerche Mat.}, 6:125--149, 1957.

\bibitem{Campanato1959}
Sergio Campanato.
\newblock Sui problemi al contorno per sistemi di equazioni differenziali
  lineari del tipo dell'elasticit\`a. {I}.
\newblock {\em Ann. Scuola Norm. Sup. Pisa Cl. Sci. (3)}, 13:223--258, 1959.

\bibitem{Campanato1959a}
Sergio Campanato.
\newblock Sui problemi al contorno per sistemi di equazioni differenziali
  lineari del tipo dell'elasticit\`a. {II}.
\newblock {\em Ann. Scuola Norm. Sup. Pisa Cl. Sci. (3)}, 13:275--302, 1959.

\bibitem{Crandall-Ishii-Lions1992}
Michael.~G. Crandall, Hitoshi Ishii, and Pierre-Louis Lions.
\newblock User's guide to viscosity solutions of second order partial
  differential equations.
\newblock {\em Bull. Amer. Math. Soc. (N.S.)}, 27(1):1--67, 1992.

\bibitem{Gonzalo-Felmer-Quaas2009}
Gonzalo D\'{a}vila, Patricio Felmer, and Alexander Quaas.
\newblock Alexandroff-{B}akelman-{P}ucci estimate for singular or degenerate
  fully nonlinear elliptic equations.
\newblock {\em C. R. Math. Acad. Sci. Paris}, 347(19-20):1165--1168, 2009.

\bibitem{GT2001}
David Gilbarg and Neil~S. Trudinger.
\newblock {\em Elliptic partial differential equations of second order}.
\newblock Classics in Mathematics. Springer-Verlag, Berlin, 2001.
\newblock Reprint of the 1998 edition.

\bibitem{Iliin-Shismarev1961}
Vladimir~A. Il'in and Il'ya~A. \v{S}i\v{s}marev.
\newblock The method of potentials for the problems of {D}irichlet and
  {N}eumann in the case of equations with discontinuous coefficients.
\newblock {\em Sibirsk. Mat. \v{Z}.}, pages 46--58, 1961.

\bibitem{Imbert-Silvestre2013}
Cyril Imbert and Luis Silvestre.
\newblock {$C^{1,\alpha}$} regularity of solutions of some degenerate fully
  non-linear elliptic equations.
\newblock {\em Adv. Math.}, 233:196--206, 2013.

\bibitem{imbert_silvestre_jems}
Cyril Imbert and Luis Silvestre.
\newblock Estimates on elliptic equations that hold only where the gradient is
  large.
\newblock {\em J. Eur. Math. Soc. (JEMS)}, 18(6):1321--1338, 2016.

\bibitem{Koike2006}
Shigeaki Koike.
\newblock {\em A beginner's guide to the theory of viscosity solutions},
  volume~13 of {\em MSJ Memoirs}.
\newblock Mathematical Society of Japan, Tokyo, 2004.

\bibitem{Li-Nirenberg2003}
YanYan Li and Louis Nirenberg.
\newblock Estimates for elliptic systems from composite material.
\newblock volume~56, pages 892--925. 2003.
\newblock Dedicated to the memory of J\"{u}rgen K. Moser.

\bibitem{Li-Vogelius2000}
YanYan Li and Michael Vogelius.
\newblock Gradient estimates for solutions to divergence form elliptic
  equations with discontinuous coefficients.
\newblock {\em Arch. Ration. Mech. Anal.}, 153(2):91--151, 2000.

\bibitem{Lions1955}
Jacques-Louis Lions and Laurent Schwartz.
\newblock Probl\`emes aux limites sur des espaces fibr\'{e}s.
\newblock {\em Acta Math.}, 94:155--159, 1955.

\bibitem{mooney}
Connor Mooney.
\newblock Harnack inequality for degenerate and singular elliptic equations
  with unbounded drift.
\newblock {\em J. Differential Equations}, 258(5):1577--1591, 2015.

\bibitem{Oleinik1961}
Olga~A. Ole\u{\i}nik.
\newblock Boundary-value problems for linear equations of elliptic parabolic
  type with discontinuous coefficients.
\newblock {\em Izv. Akad. Nauk SSSR Ser. Mat.}, 25:3--20, 1961.

\bibitem{Picone1954}
Mauro Picone.
\newblock Sur un probl{\`e}me nouveau pour l'{\'e}quation lin{\'e}aire aux
  d{\'e}riv{\'e}es partielles de la th{\'e}orie math{\'e}matique classique de
  l'{\'e}lasticit{\'e}.
\newblock In {\em Colloque sur les {\'e}quations aux d{\'e}riv{\'e}es
  partielles, CBRM, Bruxelles}, pages 9--11, 1954.

\bibitem{Schechter1960}
Martin Schechter.
\newblock A generalization of the problem of transmission.
\newblock {\em Ann. Scuola Norm. Sup. Pisa Cl. Sci. (3)}, 14:207--236, 1960.

\bibitem{Stampacchia1956}
Guido Stampacchia.
\newblock Su un problema relativo alle equazioni di tipo ellittico del secondo
  ordine.
\newblock {\em Ricerche Mat.}, 5:3--24, 1956.

\bibitem{Sheftel1963}
Zinovi~G. \v{S}eftel'.
\newblock Estimates in {$L\sb{p}$} of solutions of elliptic equations with
  discontinuous coefficients and satisfying general boundary conditions and
  conjugacy conditions.
\newblock {\em Soviet Math. Dokl.}, 4:321--324, 1963.

\end{thebibliography}
\bibliographystyle{plain}

\bigskip

\noindent\textsc{Gerardo  Huaroto}\\
Department of Mathematics\\
Federal University of Alagoas -- IM -- UFAL\\
57072-900, Cidade Universit\'aria, Mace\'o-Al, Brazil\\
\noindent\texttt{gerardo.cardenas@im.ufal.br}

\vspace{.15in}

\noindent\textsc{Edgard A. Pimentel}\\
University of Coimbra\\
CMUC, Department of Mathematics\\ 
3001-501 Coimbra, Portugal\\
and\\
Pontifical Catholic University of Rio de Janeiro -- PUC-Rio\\
22451-900, G\'avea, Rio de Janeiro-RJ, Brazil\\
\noindent\texttt{edgard.pimentel@mat.uc.pt}

\vspace{.15in}

\noindent\textsc{Giane C. Rampasso}\\
Department of Mathematics\\
University of Campinas -- IMECC -- Unicamp\\
13083-859, Cidade Universit\'aria, Campinas-SP, Brazil\\
\noindent\texttt{girampasso@ime.unicamp.br}

\vspace{.15in}

\noindent\textsc{Andrzej \'{S}wi\k{e}ch}\\
School of Mathematics\\
Georgia Institute of Technology\\
Atlanta GA 30332 USA\\
\noindent\texttt{swiech@math.gatech.edu}

\bigskip

\end{document}